\documentclass{article}
\usepackage{amssymb}
\usepackage{amsmath}
\usepackage{amsfonts}
\usepackage{geometry}

\input{tcilatex}
\begin{document}

\title{Complete left-invariant affine structures on solvable non-unimodular
three-dimensional Lie groups}
\author{Mohammed Guediri and Kholoud Al-Balawi}
\maketitle

\begin{abstract}
In this paper, we shall use a method based on the theory of extensions of
left-symmetric algebras to classify complete left-invariant affine real
structures on solvable non-unimodular three-dimensional Lie groups.
\end{abstract}

\renewcommand{\thefootnote}{} \footnotetext{%
This project was supported by King Saud University, Deanship of Scientific
Research, College of Science Research Center.
\par
{}}

\section{Introduction}

The notion of a left-symmetric algebra appeared for the first time in the
work of Koszul \cite{koszul} and Vinberg \cite{vinberg} concerning bounded
homogeneous domains and convex homogeneous cones, respectively. Over the
field of real numbers, left-symmetric algebras are of special interest
because of their role in the differential geometry of affine manifolds
(i.e., smooth manifolds with flat torsion-free affine connections), and in
the representation theory of Lie groups (see \cite{milnor2} and \cite{segal}%
). In fact, for a given simply connected Lie group $G$ with Lie algebra $%
\mathcal{G}$, the left-invariant affine structures on $G$ are in one-to-one
correspondence with the left-symmetric structures on $\mathcal{G}$
compatible with the Lie structure \cite{kim1}.

On the other hand, it is well known that there is a one-to-one
correspondence between left-invariant affine structures on a Lie group $G$
and locally simply transitive affine actions of $G$ on an $n$-dimensional
real vector space $V$ (see \cite{kim1}). The classification of
left-invariant affine structures on a given Lie group $G$ is then reduced to
the classification of compatible left-symmetric products on the Lie algebra $%
\mathcal{G}$ of $G.$ It has been proved in \cite{auslander} that a simply
connected Lie group $G$ which acts simply transitively on $\mathbb{R}^{n}$
by affine transformations is necessarily solvable. Since a few years, there
has been a growing interest in the study of simply transitive affine actions
of Lie groups on $\mathbb{R}^{n}.$ This interest is mostly due to the
example of Benoist \cite{benoist}, who constructed a simply connected
nilpotent Lie group not admitting any locally simply transitive affine
action on $\mathbb{R}^{n}.$ This example provided a negative answer to the
following question of Milnor \cite{milnor2}: Does any simply connected
solvable Lie group admit a simply transitive affine action on $\mathbb{R}%
^{n}?$

From another point of view, there is also the question of classifying all
simply transitive affine actions of a given solvable Lie group $G$ admitting
such an action. This question, even in the abelian case $G=\mathbb{R}^{k},$
seems to be very hard. When $G$ is nilpotent, the classification has been
completely achieved up to dimension four (\cite{fried-gold} and \cite{kim1}).

\medskip 

Recently, a method based on the theory of extensions of left-symmetric
algebras has been proposed in \cite{oscillator} to classify complete
left-invariant affine real structures on a given solvable Lie group of low
dimension. Since the classification in the case of solvable unimodular Lie
groups of dimension three was obtained in \cite{fried-gold}, we will use
that method to carry out in this paper the classification of complete
left-invariant affine structures on three-dimensional solvable
non-unimodular Lie groups.

\medskip 

The paper is organized as follows. In section 2, we will briefly recall some
necessary definitions and basic results on left-symmetric algebras and their
extensions. In section 3, using the classification of the three-dimensional
complex simple left-symmetric algebras given in \cite{burde} and a result in 
\cite{kong-bai-meng}, we shall first show that any complete real
left-symmetric algebra $A_{3}$ of dimension 3 whose Lie algebra is solvable
and non-unimodular is not simple. Therefore, we can get $A_{3}$ as an
extension of complete left-symmetric algebras. By using the Lie group
exponential maps, we shall deduce the classification of all complete
left-invariant affine structures on solvable non-unimodular Lie groups of
dimension 3 in terms of simply transitive actions of subgroups of the affine
group $Aff(\mathbb{R}^{3})=GL\left( \mathbb{R}^{3}\right) \mathbb{o}\mathbb{R%
}^{3}$ (see Theorem \ref{sim action}).

\smallskip

Throughout this paper, all considered vector spaces, Lie algebras, and
left-symmetric algebras are supposed to be over the field $\mathbb{R}.$ We
shall also suppose that all considered Lie groups are simply connected.

\section{Left-symmetric algebras and their extensions}

Let $A$ be a finite-dimensional vector space over $\mathbb{R}$. A
left-symmetric product on $A$ is a bilinear product that we denote by $%
x\cdot y$ satisfying 
\begin{equation}
\left( x\cdot y\right) \cdot z-\left( y\cdot x\right) \cdot z=x\cdot \left(
y\cdot z\right) -y\cdot \left( x\cdot z\right) ,  \label{l.s. def}
\end{equation}
for all $x,y,z\in A$. In this case, $A$ together with a left-symmetric
product is called left-symmetric algebra.

Now if $A$ is a left-symmetric algebra, then the commutator 
\begin{equation}
\left[ x,y\right] =x\cdot y-y\cdot x  \label{lie bracket}
\end{equation}
defines a structure of Lie algebra on $A$, called the associated Lie
algebra. On the other hand, if $\mathcal{G}$ is a Lie algebra with a
left-symmetric product $\cdot $ satisfying (\ref{lie bracket})$,$ then we
say that this left-symmetric structure is compatible with the Lie structure
on $\mathcal{G}$.

Let $G$ be a simply connected Lie group with a left-invariant affine
connection $\nabla .$ Define a product $\cdot $ on the Lie algebra $\mathcal{%
G}$ of $G$ by 
\begin{equation*}
x\cdot y=\nabla _{x}y,
\end{equation*}
for all $x,y\in \mathcal{G}$. Then, the flat and torsion-free conditions on $%
\nabla $ correspond to conditions (\ref{l.s. def}) and (\ref{lie bracket}),
respectively.

Conversely, If $G$ is a simply connected Lie group with Lie algebra $%
\mathcal{G}$ and $x\cdot y$ denotes a left-symmetric product on $\mathcal{G}$
compatible with the Lie bracket, then the left-invariant connection given by 
$\nabla _{x}y=x\cdot y$ defines a left-invariant affine structure $\nabla $
on $G.$ We deduce that if $G$ is a simply connected Lie group with Lie
algebra $\mathcal{G}$, then the study of left-invariant affine structures on 
$G$ is equivalent to the study of left-symmetric structures on $\mathcal{G}$
compatible with the Lie structure.

Let $A$ be a left-symmetric algebra whose associated Lie algebra is $%
\mathcal{G}$, and let $L_{x}$ and $R_{x}$ denote the left and right
multiplications, respectively i.e., $L_{x}y=x\cdot y$ and $R_{x}y=y\cdot x$.
The identity in (\ref{l.s. def}) is now equivalent to the formula 
\begin{equation*}
\left[ L_{x},L_{y}\right] =L_{\left[ x,y\right] },\text{\quad for all }%
x,y\in A,
\end{equation*}
or, in other words, the linear map $L:\mathcal{G\rightarrow }End\left(
A\right) $ is a representation of Lie algebras.

If a left-symmetric algebra $A$ has no proper two-sided ideal and it is not
the zero algebra of dimension $1$, then $A$ is called simple. $A$ is called
semisimple, if it is a direct sum of simple left-symmetric algebras.

We say that $A$ is complete if $R_{x}$ is a nilpotent operator for all $x\in
A.$ It turns out that, for a given simply connected Lie group $G$ with Lie
algebra $\mathcal{G}$, the complete left-invariant affine structures on $G$
are in one-to-one correspondence with the complete left-symmetric structures
on $\mathcal{G}$ compatible with the Lie structure. It is also known that an 
$n$-dimensional simply connected Lie group admits a complete left-invariant
affine structure if and only if it acts simply transitively on $\mathbb{R}%
^{n}$ by affine transformations (see \cite{kim1}). A simply connected Lie
group which is acting simply transitively on $\mathbb{R}^{n}$ by affine
transformations must be solvable according to \cite{auslander}. It is well
known that not every solvable (even nilpotent) Lie group can admit an affine
structure (see \cite{benoist}).

We say that $A$ is a Novikov algebra if it satisfies the identity 
\begin{equation}
\left( x\cdot y\right) \cdot z=\left( x\cdot z\right) \cdot y,\text{\quad
for all }x,y,z\in A.  \label{Novikov cond.}
\end{equation}

In terms of left and right multiplications, (\ref{Novikov cond.}) is
equivalent to the formula 
\begin{equation*}
\left[ R_{x},R_{y}\right] =0,\text{\quad for all }x,y\in A.
\end{equation*}

The left-symmetric algebra $A$ is called a derivation algebra if it
satisfies the identity 
\begin{equation*}
\left( x\cdot y\right) \cdot z=\left( z\cdot y\right) \cdot x,\text{\quad
for all }x,y,z\in A,
\end{equation*}
or, equivalently, all left and right multiplications $L_{x}$ and $R_{x}$ are
derivations of $\mathcal{G}$.

\medskip Recall that a Lie algebra $\widetilde{\mathcal{G}}$ is an extension
of the Lie algebra $\mathcal{G}$ by the Lie algebra $\mathcal{A}$ if there
exists a short exact sequence of Lie algebras 
\begin{equation*}
0\overset{}{\rightarrow }\mathcal{A}\overset{i}{\rightarrow }\widetilde{%
\mathcal{G}}\overset{\pi }{\rightarrow }\mathcal{G}\overset{}{\rightarrow }0.
\end{equation*}

In other words, $\mathcal{A}$ is an ideal of $\widetilde{\mathcal{G}}$ such
that $\widetilde{\mathcal{G}}/\mathcal{A}$ $\cong \mathcal{G}$.

For $\left( x,a\right) $ and $\left( y,b\right) $ in $\widetilde{\mathcal{G}}
$ $\cong \mathcal{G}\oplus \mathcal{A}$, the extended Lie bracket is given
by 
\begin{equation}
\left[ \left( x,a\right) ,\left( y,b\right) \right] =\left( \left[ x,y\right]
,\left[ a,b\right] +\phi \left( x\right) b-\phi \left( y\right) a+\omega
\left( x,y\right) \right) ,  \label{ext Lie bracket}
\end{equation}
where $\phi :\mathcal{G\rightarrow }Der\left( \mathcal{A}\right) $ is a
linear map and $\omega :\mathcal{G\times G}\rightarrow \mathcal{A}$ is an
alternating bilinear map such that 
\begin{equation*}
\left[ \phi \left( x\right) ,\phi \left( y\right) \right] =\phi \left( \left[
x,y\right] \right) +ad_{\omega \left( x,y\right) },
\end{equation*}
and 
\begin{equation*}
\omega \left( \left[ x,y\right] ,z\right) -\omega \left( x,\left[ y,z\right]
\right) +\omega \left( y,\left[ x,z\right] \right) =\phi \left( x\right)
\omega \left( y,z\right) +\phi \left( y\right) \omega \left( z,x\right)
+\phi \left( z\right) \omega \left( x,y\right) .
\end{equation*}

Note here that if $\mathcal{A}$ is abelian, then $\omega $ is a 2-cocycle.
(For more details, we refer to \cite{neeb} and \cite{jacobson}).

Now we shall briefly discuss the problem of extension of a left-symmetric
algebra by another left-symmetric algebra. To our knowledge, the notion of
extensions of left-symmetric algebras has been considered for the first time
in \cite{kim1}, to which we refer the reader for more details. See also \cite%
{chang-kim}.

Suppose that a vector space extension $\widetilde{A}$ of a left-symmetric
algebra $A$ by another left-symmetric algebra $E$ is given. We want to
define a left-symmetric structure on $\widetilde{A}$ in terms of the
left-symmetric structures given on $A$ and $E.$ In other words, we want to
define a left-symmetric product on $\widetilde{A}$ for which $E$ becomes a
two-sided ideal in $\widetilde{A}$ such that $\widetilde{A}/E$ $\cong A$; or
equivalently, 
\begin{equation*}
0\overset{}{\rightarrow }E\overset{}{\rightarrow }\widetilde{A}\overset{}{%
\rightarrow }A\overset{}{\rightarrow }0
\end{equation*}
becomes a short exact sequence of left-symmetric algebras.

\begin{theorem}[\protect\cite{kim1}]
\label{kim}There exists a left-symmetric structure on $\widetilde{A}$
extending a left-symmetric algebra $A$ by a left-symmetric algebra $E$ if
and only if there exist two linear maps $\lambda ,\rho :A\rightarrow $ $%
End(E)$ and a bilinear map $g:A\times A\rightarrow E$ such that, for all $%
x,y,z\in A$ and $a,b\in E,$ the following conditions are satisfied.

\begin{enumerate}
\item $\lambda _{x}\left( a\cdot b\right) =\lambda _{x}(a)\cdot b+a\cdot
\lambda _{x}(b)-\rho _{x}(a)\cdot b,$

\item $\rho _{x}\left( \left[ a,b\right] \right) =a\cdot \rho _{x}(b)-b\cdot
\rho _{x}(a),$

\item $[\lambda _{x},\lambda _{y}]-\lambda _{[x,y]}=L_{g(x,y)-g(y,x)},$

\item $[\lambda _{x},\rho _{y}]+$ $\rho _{y}\circ \rho _{x}-\rho _{x\cdot
y}=R_{g(x,y)}$

\item $g\left( x,y\cdot z\right) -g\left( y,x\cdot z\right) $ $+\lambda
_{x}\left( g\left( y,z\right) \right) -\lambda _{y}$ $\left( g\left(
x,z\right) \right) -$ $g\left( \left[ x,y\right] ,z\right) $

$-\rho _{z}\left( g\left( x,y\right) -g\left( y,x\right) \right) =0.$
\end{enumerate}
\end{theorem}

If the conditions of the above theorem are fulfilled, then the extended
left-symmetric product on $\widetilde{A}\cong A\times E$ is given by 
\begin{equation}
\left( x,a\right) \cdot \left( y,b\right) =\left( x\cdot y,a\cdot b+\lambda
_{x}\left( b\right) +\rho _{y}\left( a\right) +g\left( x,y\right) \right) .
\label{product}
\end{equation}

It is remarkable that if the left-symmetric product of $E$ is trivial, then
the conditions of the above theorem simplify to the following three
conditions:

\begin{description}
\item[(i)] $\left[ \lambda _{x},\lambda _{y}\right] =\lambda _{\left[ x,y%
\right] },$ i.e. $\lambda $ is a representation of Lie algebras,

\item[(ii)] $\left[ \lambda _{x},\rho _{y}\right] =\rho _{x\cdot y}-\rho
_{y}\circ \rho _{x}.$

\item[(iii)] $g\left( x,y\cdot z\right) -g\left( y,x\cdot z\right) +\lambda
_{x}\left( g\left( y,z\right) \right) -\lambda _{y}\left( g\left( x,z\right)
\right) -g\left( \left[ x,y\right] ,z\right) $

\item $-\rho _{z}\left( g\left( x,y\right) -g\left( y,x\right) \right) =0.$
\end{description}

\smallskip In this case, $E$ becomes a $A$-bimodule and the extended product
given in (\ref{product}) simplifies too. Recall that if $K$ is a
left-symmetric algebra and $V$ is a vector space, then we say that $V$ is a $%
K$-bimodule if there exist two linear maps $\lambda ,$ $\rho :K\rightarrow
End\left( V\right) $ which satisfy the conditions (i) and (ii) stated above.

\bigskip

Let $K$ be a left-symmetric algebra, and suppose that a $K$-bimodule $V$ is
known. We denote by $L^{p}\left( K,V\right) $ the space of all $p$-linear
maps from $K$ to $V$, and we define two coboundary operators $\delta
_{1}:L^{1}\left( K,V\right) \rightarrow L^{2}\left( K,V\right) $ and $\delta
_{2}:L^{2}\left( K,V\right) \rightarrow L^{3}\left( K,V\right) $ as follows:

For a linear map $h\in L^{1}\left( K,V\right) $ we set 
\begin{equation}
\delta _{1}h\left( x,y\right) =\rho _{y}\left( h\left( x\right) \right)
+\lambda _{x}\left( h\left( y\right) \right) -h\left( x\cdot y\right) ,
\label{cob1}
\end{equation}
and for a bilinear map $g\in L^{2}\left( K,V\right) $ we set

\begin{eqnarray}
\delta _{2}g\left( x,y,z\right) &=&g\left( x,y\cdot z\right) -g\left(
y,x\cdot z\right) +\lambda _{x}\left( g\left( y,z\right) \right) -\lambda
_{y}\left( g\left( x,z\right) \right)  \label{cob2} \\
&&-g\left( \left[ x,y\right] ,z\right) -\rho _{z}\left( g\left( x,y\right)
-g\left( y,x\right) \right)  \notag
\end{eqnarray}
where $\lambda $ and $\rho $ are linear maps $\lambda ,$ $\rho :K\rightarrow
End\left( V\right) .$

It is straightforward to check that $\delta _{2}\circ \delta _{1}=0.$
Therefore, if we set $Z_{\lambda ,\rho }^{2}\left( K,V\right) =\ker \delta
_{2}$ and $B_{\lambda ,\rho }^{2}\left( K,V\right) =\func{Im}\delta _{1},$
we can define a notion of second cohomology for the actions $\lambda $ and $%
\rho $ by simply setting $H_{\lambda ,\rho }^{2}\left( K,V\right)
=Z_{\lambda ,\rho }^{2}\left( K,V\right) /B_{\lambda ,\rho }^{2}\left(
K,V\right) .$ As in the case of Lie algebras, we can prove the following
(see \cite{kim1}).

\begin{proposition}
\label{class}For given linear maps $\lambda ,$ $\rho :K\rightarrow End\left(
V\right) ,$ the equivalent classes of extensions 
\begin{equation*}
0\rightarrow V\rightarrow A\rightarrow K\rightarrow 0
\end{equation*}
of $K$ by $V$ are in one-to-one correspondence with the elements of the
second cohomology group $H_{\lambda ,\rho }^{2}\left( K,V\right) .$
\end{proposition}

A left-symmetric algebras extension

\begin{equation*}
0\overset{}{\rightarrow }E\overset{i}{\rightarrow }\widetilde{A}\overset{\pi 
}{\rightarrow }A\overset{}{\rightarrow }0
\end{equation*}
is called central if and only if $i\left( E\right) \subseteq C\left( 
\widetilde{A}\right) $ where 
\begin{equation*}
C\left( \widetilde{A}\right) =\left\{ x\in \widetilde{A}:x\cdot y=y\cdot
x=0\right\}
\end{equation*}
is the center of $\widetilde{A}.$ In particular, the extension is central
whenever $E$ is a trivial $A$-bimodule (i.e., $\lambda =\rho =0$).

We say that the extension is exact if and only if $i\left( E\right) =C\left( 
\widetilde{A}\right) .$ It is easy to verify (see \cite{kim1}) that the
extension is exact if and only if $I_{[g]}=0,$ where

\begin{equation*}
I_{[g]}=\left\{ x\in A:x\cdot y=y\cdot x=0\text{ and }g(x,y)=g(y,x)=0\text{
for all }y\in A\right\}
\end{equation*}

We observe that $I_{[g]}$ is depends only on the cohomology class of $g,$
that is $I_{[g]}$ is well defined.

In case $E$ is a trivial $A$-bimodule, we denote the central extension
corresponding to the class $[g]\in H^{2}\left( A,E\right) $ by $\left( 
\widetilde{A},\left[ g\right] \right) .$

Let $\left( \widetilde{A},\left[ g\right] \right) $ and $\left( \widetilde{A}%
^{^{\prime }},\left[ g^{^{\prime }}\right] \right) $ be two central
extensions of $A$ by $E,$ and $\mu \in Aut\left( E\right) =GL\left( E\right) 
$ and $\eta \in Aut\left( A\right) ,$ where $Aut\left( E\right) $ and $%
Aut\left( A\right) $ are the groups of left-symmetric automorphisms of $E$
and $K,$ respectively. It is clear that if, $h\in L^{1}\left( A,E\right) ,$
then the linear mapping $\psi :\widetilde{A}\rightarrow \widetilde{A}%
^{^{\prime }}$ defined by 
\begin{equation*}
\psi \left( x,a\right) =\left( \eta \left( x\right) ,\mu \left( a\right)
+h\left( x\right) \right)
\end{equation*}
is an isomorphism provided $g^{^{\prime }}\left( \eta \left( x\right) ,\eta
\left( y\right) \right) =\mu \left( g\left( x,y\right) \right) +\delta
_{1}h\left( x,y\right) $ for all $\left( x,y\right) \in A\times A,$ i.e., $%
\eta ^{*}\left[ g^{^{\prime }}\right] =\mu _{*}\left[ g\right] .$

This allows us to define an action of the group $G=Aut\left( E\right) \times
Aut\left( A\right) $ on $H^{2}\left( A,E\right) $ by setting 
\begin{equation*}
\left( \mu ,\eta \right) \cdot \left[ g\right] =\mu _{*}\eta ^{*}\left[ g%
\right]
\end{equation*}
or equivalently, $\left( \mu ,\eta \right) \cdot g\left( x,y\right) =\mu
\left( g\left( \eta \left( x\right) ,\eta \left( y\right) \right) \right) $
for all $x,y\in A.$

Denoting the set of all exact central extensions of $A$ by $E$ by 
\begin{equation*}
H_{ex}^{2}\left( A,E\right) =\left\{ \left[ g\right] \in H^{2}\left(
A,E\right) :I_{[g]}=0\right\}
\end{equation*}%
and the orbit of $\left[ g\right] $ by $G_{\left[ g\right] },$ it turns out
that the following result is valid (see \cite{kim1}).

\begin{proposition}
\label{isomo} Let $\left[ g\right] $ and $\left[ g^{^{\prime }}\right] $ be
two classes in $H_{ex}^{2}\left( A,E\right) .$ Then, the central extensions $%
\left( \widetilde{A},\left[ g\right] \right) $ and $\left( \widetilde{A}%
^{^{\prime }},\left[ g^{^{\prime }}\right] \right) $ are isomorphic if and
only if $G_{\left[ g\right] }=G_{\left[ g^{^{\prime }}\right] }.$ In other
words, the classification of the exact central extensions of $A$ by $E$ is,
up to left-symmetric isomorphism, the orbit space of $H_{ex}^{2}\left(
A,E\right) $ under the natural action of $G=Aut\left( E\right) \times
Aut\left( A\right) .$
\end{proposition}

We close this section by the following important result (compare to \cite%
{chang-kim})

\begin{proposition}
\label{completeness}Let $0\rightarrow I\overset{}{\rightarrow }A\overset{}{%
\rightarrow }J\rightarrow 0$ be an exact sequence of left-symmetric algebras
such that $A$ is complete. Then, $I$ and $J$ are complete.
\end{proposition}

\begin{proof}
Let $A$ be a complete left-symmetric algebra. Then $R_{x}$ is nilpotent for
all $x\in A.$ Since $I$ is an ideal of $A$, then $R_{x}$ is nilpotent for
all $x\in I,$ that is $I$ is complete. On the other hand, Since $J\cong A/I,$
we can define for $x\in A,$ $R_{x}\mid _{J}:J\rightarrow J,$ by $R_{x}\mid
_{J}\left( \overline{y}\right) =R_{x}y+I$ for all $y\in A,$ $\overline{y}%
=y+I.$ Since for all $y_{1},y_{2}\in A$ such that $y_{1}+I=y_{2}+I$ there
exists $z\in I$ so that $y_{2}=y_{1}+z,$ and 
\begin{eqnarray*}
R_{x}\left( y_{2}+I\right)  &=&R_{x}y_{2}+I \\
&=&R_{x}\left( y_{1}+z\right) +I \\
&=&R_{x}y_{1}+R_{x}z+I \\
&=&R_{x}y_{1}+I \\
&=&R_{x}\left( y_{1}+I\right) 
\end{eqnarray*}%
then, $R_{x}\mid _{J}$ is well defined. We also have, for all $x,y\in A,$
that%
\begin{eqnarray*}
R_{\overline{x}}\overline{y} &=&\left( y+I\right) \cdot \left( x+I\right)  \\
&=&y\cdot x+I \\
&=&R_{x}y+I \\
&=&R_{x}\overline{y}
\end{eqnarray*}

Thus, to prove that $J$ is complete, it is enough to prove that $R_{x}\mid
_{J}$ is nilpotent for all $x\in A$. Since $R_{x}$ is nilpotent, then $%
R_{x}^{k}=0$ for some $k\in \mathbb{N}$. This implies that 
\begin{equation*}
R_{x}^{k}\left( y\right) +I=I=\overline{0}
\end{equation*}%
for all $y\in A.$ Hence, $R_{x}^{k}\left( \overline{y}\right) =0$ for all $%
\overline{y}\in J,$ that is $R_{x}\mid _{J}$ is nilpotent for all $x\in A,$
and hence $J$ is complete.
\end{proof}

\section{Complete left-symmetric structures on solvable non-unimodular Lie
algebras of dimension 3}

Recall that a Lie algebra $\mathcal{G}$ is unimodular if and only if tr$%
\left( ad_{x}\right) =0$ for all $x\in \mathcal{G}$. The classification of
solvable non-unimodular Lie algebras of dimension 3 can be found in \cite%
{milnor1}.

\begin{lemma}
\label{milnorj}Let $\mathcal{G}$ be a solvable non-unimodular Lie algebra of
dimension 3. Then, there is a basis $\left\{ e_{1},e_{2},e_{3}\right\} $ of $%
\mathcal{G}$ so that 
\begin{eqnarray*}
\left[ e_{1},e_{2}\right] &=&\alpha e_{2}+\beta e_{3} \\
\left[ e_{1},e_{3}\right] &=&\gamma e_{2}+(2-\alpha )e_{3}
\end{eqnarray*}

If we exclude the case where $D$ is the identity matrix, then the
determinant $\det D=\alpha (2-\alpha )-\beta \gamma $ provides a complete
isomorphism invariant for this Lie algebra.
\end{lemma}

According to this result, we can, by simple computations, find that there
are five possibilities for $D:$%
\begin{eqnarray*}
D &\cong &\left( 
\begin{array}{ll}
0 & 0 \\ 
0 & 1%
\end{array}
\right) ,\qquad D\cong \left( 
\begin{array}{ll}
1 & 0 \\ 
0 & 1%
\end{array}
\right) ,\qquad D\cong \left( 
\begin{array}{ll}
1 & 0 \\ 
1 & 1%
\end{array}
\right) \\
D &\cong &\left( 
\begin{array}{ll}
1 & 0 \\ 
0 & \mu%
\end{array}
\right) ,\text{ where }0<\left| \mu \right| <1\text{ or }D\cong \left( 
\begin{array}{ll}
1 & -\zeta \\ 
\zeta & 1%
\end{array}
\right) \text{ where }\zeta >0
\end{eqnarray*}

This implies that any solvable non-unimodular Lie algebra of dimension 3 is
isomorphic to one and only one of the following Lie algebras

\begin{equation*}
\begin{tabular}{l}
$\mathcal{G}_{3,1}$:$\quad \left[ e_{1},e_{2}\right] =e_{2}$ \\ 
$\mathcal{G}_{3,2}$:$\quad \left[ e_{1},e_{2}\right] =e_{2},$ $\left[
e_{1},e_{3}\right] =e_{3}$ \\ 
$\mathcal{G}_{3,3}$:$\quad \left[ e_{1},e_{2}\right] =e_{2}+e_{3},$ $\left[
e_{1},e_{3}\right] =e_{3}$ \\ 
$\mathcal{G}_{3,4}^{\mu }$:$\quad \left[ e_{1},e_{2}\right] =e_{2},$ $\left[
e_{1},e_{3}\right] =\mu e_{3},$ $0<\left| \mu \right| <1$ \\ 
$\mathcal{G}_{3,5}^{\zeta }$:$\quad \left[ e_{1},e_{2}\right] =e_{2}+\zeta
e_{3},$ $\left[ e_{1},e_{3}\right] =-\zeta e_{2}+e_{3},$ $\zeta >0$%
\end{tabular}%
\end{equation*}

Now let $\mathcal{G}$ be a real solvable non-unimodular Lie algebra of
dimension 3. Let $A_{3}$ be a complete left-symmetric algebra whose
associated Lie algebra is $\mathcal{G}$.

We shall first recall the following result from \cite{kong-bai-meng}.

\begin{lemma}
\label{complexification}Only the complex simple left-symmetric algebras and
even-dimensional complex semisimple left-symmetric algebras may have simple
real forms, where a real form of a complex left-symmetric algebra $A$ is a
subalgebra $A_{0}$ of $A^{\mathbb{R}}$ such that $A_{0}^{\mathbb{C}}=A.$
Here $A^{\mathbb{R}}$ is $A$ regarded as a real left-symmetric algebra.
\end{lemma}

Now, we can prove the following

\begin{proposition}
\label{p1}$A_{3}$ is not simple. In other words, any complete left-symmetric
structure on a solvable non-unimodular Lie algebra of dimension 3 is not
simple.
\end{proposition}

%TCIMACRO{\TeXButton{Proof}{\proof}}%
%BeginExpansion
\proof%
%EndExpansion
Assume to the contrary that $A_{3}$ is simple. Then, Lemma \ref%
{complexification} shows that the complexification $A_{3}^{\mathbb{C}}$ of $%
A_{3}$ is simple as the dimension of $A_{3}^{\mathbb{C}}$ is odd. We can now
apply Corollary 4.2 in \cite{burde} to deduce that $A_{3}^{\mathbb{C}}$ is
isomorphic to the complex left-symmetric algebra $A_{1}^{-1}$ having a basis 
$\left\{ e_{1},e_{2},e_{3}\right\} $ such that the only non-trivial products
are

\begin{eqnarray*}
e_{1}\cdot e_{2} &=&e_{2}, \\
e_{1}\cdot e_{3} &=&-e_{3}, \\
e_{2}\cdot e_{3} &=&e_{3}\cdot e_{2}=e_{1}.
\end{eqnarray*}

Thus, the complex Lie algebra $\mathcal{G}_{3}$ associated to $A_{3}^{%
\mathbb{C}}\cong A_{1}^{-1}$ is unimodular and hence $\mathcal{G}$ must be
unimodular. This contradiction showsshows that $A_{3}$ is not simple%
%TCIMACRO{\TeXButton{End Proof}{\endproof}}%
%BeginExpansion
\endproof%
%EndExpansion

\bigskip

Before returning to the left-symmetric algebra $A_{3},$ we need to state the
following facts without proofs.

\begin{lemma}
\label{ideal}Let $A$ be a left-symmetric algebra with associated Lie algebra 
$\mathcal{G}$, and $R$ a two-sided ideal in $A.$ Then, the Lie algebra $%
\mathcal{R}$ associated to $R$ is an ideal in $\mathcal{G}$.
\end{lemma}

\begin{lemma}
\label{ideal 2}Let $\mathcal{G}$ be a solvable non-unimodular Lie algebra of
dimension 3 and let $\mathcal{I}$ be a proper ideal of $\mathcal{G}$. Then, $%
\mathcal{I}$ is isomorphic to $\mathbb{R}$, $\mathbb{R}^{2},$ or aff$\left( 
\mathbb{R}\right) =\left\langle e_{1},e_{2}:\left[ e_{1},e_{2}\right]
=e_{2}\right\rangle .$
\end{lemma}

By Proposition \ref{p1}, $A_{3}$ is not simple and hence it has a proper
two-sided ideal $I,$ so we get a short exact sequence of left-symmetric
algebras 
\begin{equation}
0\rightarrow I\overset{i}{\rightarrow }A_{3}\overset{\pi }{\rightarrow }%
J\rightarrow 0  \label{(1)}
\end{equation}

If $\mathcal{I}$ is the Lie subalgebra associated to $I$ then, by Lemma \ref%
{ideal}, $\mathcal{I}$ is an ideal in $\mathcal{G}$. From Lemma \ref{ideal 2}
it follows that there are three cases to be considered according to whether $%
\mathcal{I}$ is isomorphic to $\mathbb{R}$, $\mathbb{R}^{2},$ or aff$\left( 
\mathbb{R}\right) .$

\begin{itemize}
\item Case 1. $\mathcal{I}\cong \mathbb{R}$.
\end{itemize}

In this case, the short exact sequence (\ref{(1)}) becomes 
\begin{equation*}
0\rightarrow \mathbb{R}_{0}\rightarrow A_{3}\rightarrow I_{2}\rightarrow 0
\end{equation*}
where $I_{2}$ is a complete left-symmetric algebra of dimension 2 and $%
\mathbb{R}_{0}$ is $\mathbb{R}$ with the trivial product.

At the Lie algebra level, we have a short exact sequence of Lie algebras of
the form 
\begin{equation}
0\rightarrow \mathbb{R}\rightarrow \overset{\sim }{\mathcal{G}}\rightarrow 
\mathcal{H}_{2}\rightarrow 0  \label{(2)}
\end{equation}
where $\mathcal{H}_{2}$ denotes the associated Lie algebra of $I_{2}$ and $%
\overset{\sim }{\mathcal{G}}$ is an extension of $\mathcal{H}_{2}$ by $%
\mathbb{R}$.

Since $\mathcal{H}_{2}$ is of dimension 2, then $\mathcal{H}_{2}$ is either
isomorphic to $\mathbb{R}^{2}$ or aff$\left( \mathbb{R}\right) .$\medskip

Assume first that $\mathcal{H}_{2}$\textbf{\ }$\cong \mathbb{R}^{2}.$ Then,
the short exact sequence (\ref{(2)}) becomes 
\begin{equation*}
0\rightarrow \mathbb{R}\rightarrow \overset{\sim }{\mathcal{G}}\rightarrow 
\mathbb{R}^{2}\rightarrow 0
\end{equation*}

Let $\left\{ e_{1},e_{2}\right\} $ be a basis for $\mathbb{R}^{2}.$ On $%
\mathbb{R}^{2}\mathbb{\times R}$, the extended Lie bracket given by (\ref%
{ext Lie bracket}) takes the simplified form 
\begin{equation}
\left[ \left( x,a\right) ,\left( y,b\right) \right] =\left( 0,\phi \left(
x\right) b-\phi \left( y\right) a+\omega \left( x,y\right) \right) ,
\label{(3)}
\end{equation}
for all $a,b\in \mathbb{R}$, $x,y\in \mathbb{R}^{2}$.

Setting $\overset{\sim }{e}_{i}=\left( e_{i},0\right) $, $i=1,2$ and $%
\overset{\sim }{e}_{3}=\left( 0,1\right) $ we get 
\begin{eqnarray*}
\left[ \widetilde{e}_{1},\widetilde{e}_{2}\right] &=&\omega \left(
e_{1},e_{2}\right) \widetilde{e}_{3} \\
\left[ \widetilde{e}_{1},\widetilde{e}_{3}\right] &=&\phi \left(
e_{1}\right) \widetilde{e}_{3} \\
\left[ \widetilde{e}_{2},\widetilde{e}_{3}\right] &=&\phi \left(
e_{2}\right) \widetilde{e}_{3}
\end{eqnarray*}

Since $\mathcal{G}$ is solvable and non-unimodular, we can, without loss of
generality, assume that $\phi \left( e_{2}\right) =0.$ That is 
\begin{equation*}
D=\left( 
\begin{array}{ll}
0 & \omega \left( e_{1},e_{2}\right) \\ 
0 & \phi \left( e_{1}\right)%
\end{array}
\right)
\end{equation*}

Notice that $\phi \left( e_{1}\right) $ should be non-zero, since otherwise $%
\mathcal{G}$ becomes unimodular. In other words, 
\begin{equation*}
D\cong \left( 
\begin{array}{ll}
0 & 0 \\ 
0 & 1%
\end{array}
\right)
\end{equation*}

Now, we shall determine all the complete left-symmetric structures on $%
\mathbb{R}^{2}$. These are described by the following lemma that we state
without proof.

\begin{lemma}
\label{r2}Up to left-symmetric isomorphism, there are two complete
left-symmetric structures on $\mathbb{R}^{2}$ given, in a basis $\left\{
e_{1},e_{2}\right\} $ of $\mathbb{R}^{2},$ by either

\begin{enumerate}
\item[(i)] $e_{i}\cdot e_{j}=0,$ $i,j=1,2$

\item[(ii)] $e_{2}\cdot e_{2}=e_{1}.$
\end{enumerate}
\end{lemma}

From now on, $A_{2}$ will denote the vector space $\mathbb{R}^{2}$ endowed
with one of the complete left-symmetric structures described in Lemma \ref%
{r2}.

The extended left-symmetric product on $A_{2}\mathbb{\times R}_{0}$ given by
(\ref{product}) turns out to take the simplified form$\allowbreak $%
\begin{equation}
\left( x,a\right) \cdot \allowbreak \left( y,b\right) =\left( x\cdot
y,b\lambda _{x}+a\rho _{y}+g\left( x,y\right) \right) ,  \label{(4)}
\end{equation}
for all $x,y\in A_{2}$ and $a,b\in \mathbb{R}$. Indeed, $\rho _{x},\lambda
_{x}\in End\left( \mathbb{R}\right) \cong \mathbb{R}$ for all $x\in A_{2}$.
So, we can identify $\rho _{x}$ and $\lambda _{x}$ with real numbers that we
denote by $\rho _{x}$ and $\lambda _{x},$ respectively.

Note here that $\lambda _{x}=\phi \left( x\right) +\rho _{x},$ for all $x\in 
\mathbb{R}^{2}$where $\phi :\mathbb{R}^{2}\rightarrow End\left( \mathbb{R}%
\right) \cong \mathbb{R}$ as in (\ref{(3)}).

The conditions in Theorem \ref{kim} can be simplified to the following
conditions

\begin{center}
\begin{equation}
\rho _{\left( x\cdot y\right) }=\rho _{y}\circ \rho _{x}  \label{id1}
\end{equation}

\begin{equation}
\begin{tabular}{l}
$g\left( x,y\cdot z\right) -g\left( y,x\cdot z\right) +\lambda _{x}\left(
g\left( y,z\right) \right) -\lambda _{y}\left( g\left( x,z\right) \right) $
\\ 
$\quad \quad \quad \quad \quad \quad -\rho _{z}\left( g\left( x,y\right)
-g\left( y,x\right) \right) =0$%
\end{tabular}
\label{id2}
\end{equation}
\end{center}

By using (\ref{(3)}) and (\ref{(4)}), we deduce from 
\begin{equation}
\left[ \left( x,a\right) ,\left( y,b\right) \right] =\left( x,a\right) \cdot
\allowbreak \left( y,b\right) -\left( y,b\right) \cdot \left( x,a\right) ,
\label{l.s cond}
\end{equation}
that 
\begin{equation*}
\omega \left( x,y\right) =g(x,y)-g(y,x)\text{ .}
\end{equation*}

Since $\omega \left( e_{1},e_{2}\right) =0,$ then $%
g(e_{1},e_{2})=g(e_{2},e_{1}).$ Since $\phi \left( e_{2}\right) =0,$ then $%
\lambda _{e_{2}}=\rho _{e_{2}}.$ Also, since $\phi \left( e_{1}\right) \neq
0,$ then $\lambda _{e_{1}}-\rho _{e_{1}}\neq 0.$ By applying identity (\ref%
{id1}) to $e_{i}\cdot e_{i},$ $i=1,2,$ we deduce that $\rho =0.$ Hence $%
\lambda _{e_{2}}=0$ and $\lambda _{e_{1}}\neq 0,$ say $\lambda
_{e_{1}}=\alpha ,$ $\alpha \in \mathbb{R}^{*}.$

In this case, the formula (\ref{cob1}) and (\ref{cob2}) become 
\begin{equation*}
\delta _{1}h\left( x,y\right) =\lambda _{x}\left( h\left( y\right) \right)
-h\left( x\cdot y\right)
\end{equation*}
and 
\begin{equation*}
\delta _{2}g\left( x,y,z\right) =g\left( x,y\cdot z\right) -g\left( y,x\cdot
z\right) +\lambda _{x}\left( g\left( y,z\right) \right) -\lambda _{y}\left(
g\left( x,z\right) \right)
\end{equation*}
where $h\in \mathcal{L}^{1}\left( A_{2},\mathbb{R}\right) $ and $g\in 
\mathcal{L}^{2}\left( A_{2},\mathbb{R}\right) .$

According to Lemma \ref{r2}, there are two cases to be considered.

\begin{enumerate}
\item $A_{2}=\left\langle e_{1},e_{2}:e_{i}\cdot
e_{j}=0,i,j=1,2\right\rangle .$

In this case, using the first formula above for $\delta _{1},$ we get 
\begin{equation*}
\delta _{1}h=\left( 
\begin{array}{cc}
h_{11} & h_{12} \\ 
0 & 0%
\end{array}
\right) ,
\end{equation*}
where $h_{11}=\alpha h\left( e_{1}\right) $ and $h_{12}=\alpha h\left(
e_{2}\right) .$ Similarly, using the second formula above for $\delta _{2},$
we verify easily that if $g$ is a cocycle (i.e. $\delta _{2}g=0$) and $%
g_{ij}=g\left( e_{i},e_{j}\right) ,$ then 
\begin{equation*}
g=\left( 
\begin{array}{cc}
g_{11} & 0 \\ 
0 & 0%
\end{array}
\right) ,
\end{equation*}
that is $g_{12}=g_{21}=g_{22}=0.$ In this case, the class $\left[ g\right]
\in H_{\lambda ,\rho }^{2}\left( A_{2},\mathbb{R}\right) $ of a cocycle $g$
may be represented, in the basis above, by a matrix of the simplified form 
\begin{equation*}
g=\left( 
\begin{array}{cc}
0 & s \\ 
0 & 0%
\end{array}
\right)
\end{equation*}

We can now determine the extended complete left-symmetric structures on $%
A_{3}.$ By setting $\widetilde{e}_{i}=\left( e_{i},0\right) $, $i=1,2$ and $%
\widetilde{e}_{3}=\left( 0,1\right) $ and using formula (\ref{(4)}) we
obtain that the non-zero relations in $A_{3}$ are 
\begin{eqnarray*}
\widetilde{e}_{1}\cdot \widetilde{e}_{2}=s\widetilde{e}_{3}, \\
\widetilde{e}_{1}\cdot \widetilde{e}_{3}=\alpha \widetilde{e}_{3},
\end{eqnarray*}
with $\alpha =\lambda _{e_{1}}\neq 0$

By setting $e_{1}=\frac{1}{\alpha }\widetilde{e}_{1},$ $e_{2}=\widetilde{e}%
_{3}$ and $e_{3}=\widetilde{e}_{2},$ and $t=\frac{s}{\alpha }$ we see that
the new basis $\left\{ e_{1},e_{2},e_{3}\right\} $ of $A_{3}$ satisfies 
\begin{eqnarray*}
e_{1}\cdot e_{2}=e_{2} \\
e_{1}\cdot e_{3}=te_{2}
\end{eqnarray*}
and all other products are zero. We can easily see that this product is
isomorphic to 
\begin{equation*}
e_{1}\cdot e_{2}=e_{2}.
\end{equation*}

We set $N_{3,0}=\left\langle e_{1},e_{2},e_{3}:e_{1}\cdot
e_{2}=e_{2}\right\rangle .$

\item $A_{2}=\left\langle e_{1},e_{2}:e_{2}\cdot e_{2}=e_{1}\right\rangle .$

We obtain, as above, that $A_{3}$ is isomorphic to one of the following
complete left-symmetric algebras

\begin{description}
\item[(i)] $N_{3,2}=\left\langle e_{1},e_{2},e_{3}:e_{1}\cdot
e_{2}=e_{2},e_{3}\cdot e_{3}=e_{1}\right\rangle ,$

\item[(ii)] $N_{3,3}=\left\langle e_{1},e_{2},e_{3}:e_{1}\cdot
e_{2}=e_{2},e_{3}\cdot e_{3}=-e_{1}\right\rangle .$
\end{description}
\end{enumerate}

Assume now that $\mathcal{H}_{2}$ $\cong $aff$\left( \mathbb{R}\right) .$
Then the extended Lie bracket on aff$\left( \mathbb{R}\right) \times \mathbb{%
R}$ given by (\ref{ext Lie bracket}) takes the form 
\begin{equation*}
\left[ \left( x,a\right) ,\left( y,b\right) \right] =\left( \left[ x,y\right]
,\phi \left( x\right) b-\phi \left( y\right) a+\omega \left( x,y\right)
\right) ,
\end{equation*}
for all $a,b\in \mathbb{R}$, $x,y\in $aff$\left( \mathbb{R}\right) $.

Let $\left\{ e_{1},e_{2}\right\} $ be a basis of aff$\left( \mathbb{R}%
\right) $ satisfying $\left[ e_{1},e_{2}\right] =e_{2}$. By setting $\overset%
{\sim }{e}_{i}=\left( e_{i},0\right) $, $i=1,2$ and $\overset{\sim }{e}%
_{3}=\left( 0,1\right) $ we get 
\begin{eqnarray*}
\left[ \widetilde{e}_{1},\widetilde{e}_{2}\right] &=&\widetilde{e}%
_{2}+\omega \left( e_{1},e_{2}\right) \widetilde{e}_{3} \\
\left[ \widetilde{e}_{1},\widetilde{e}_{3}\right] &=&\phi \left(
e_{1}\right) \widetilde{e}_{3} \\
\left[ \widetilde{e}_{2},\widetilde{e}_{3}\right] &=&\phi \left(
e_{2}\right) \widetilde{e}_{3}.
\end{eqnarray*}

Since $\mathcal{G}$ is solvable and non-unimodular, then as above, we can
assume that $\phi \left( e_{2}\right) =0.$ That is, 
\begin{equation*}
D=\left( 
\begin{array}{ll}
1 & \omega \left( e_{1},e_{2}\right) \\ 
0 & \phi \left( e_{1}\right)%
\end{array}
\right)
\end{equation*}

Notice that $\phi \left( e_{1}\right) +1\neq 0,$ since otherwise $\mathcal{G}
$ becomes unimodular. Now, we have the following cases.

\begin{enumerate}
\item If $\det D=0,$ then $D\cong \left( 
\begin{array}{ll}
1 & 0 \\ 
0 & 0%
\end{array}
\right) $ that is, $\phi \left( e_{1}\right) =0$ and $\omega \left(
e_{1},e_{2}\right) =0$. This means that $\phi \,$is identically zero, i.e., $%
\overset{\sim }{\mathcal{G}}$ is a central extension of aff$\left( \mathbb{R}%
\right) $by $\mathbb{R}$.

\item If $\det D\neq 0,$ then $D\cong \left( 
\begin{array}{ll}
1 & 0 \\ 
0 & 1%
\end{array}
\right) ,\left( 
\begin{array}{ll}
1 & 1 \\ 
0 & 1%
\end{array}
\right) $ or $\left( 
\begin{array}{ll}
1 & 0 \\ 
0 & \mu%
\end{array}
\right) ,$ with $0<\left| \mu \right| <1.$
\end{enumerate}

It is not hard to prove the following

\begin{lemma}
\label{g2}Up to left-symmetric isomorphisms, there is a unique complete
left-symmetric structure on aff$\left( \mathbb{R}\right) $ which is given,
relative to a basis $e_{1},e_{2}$ of aff$\left( \mathbb{R}\right) $
satisfying $\left[ e_{1},e_{2}\right] =e_{2},$ by $e_{1}\cdot e_{2}=e_{2}$.
\end{lemma}

We will denote by $N_{2}$ the vector space aff$\left( \mathbb{R}\right) $
endowed with the complete left-symmetric product given in Lemma \ref{g2}.

On the other hand, the extended left-symmetric product on $N_{2}\times 
\mathbb{R}_{0}$ is given by 
\begin{equation}
\left( x,a\right) \cdot \allowbreak \left( y,b\right) =\left( x\cdot
y,b\lambda \left( x\right) +a\rho \left( y\right) +g\left( x,y\right)
\right) ,\allowbreak  \label{l.s p}
\end{equation}
for all $a,b\in \mathbb{R}$, $x,y\in $aff$\left( \mathbb{R}\right) .$

The conditions in Theorem \ref{kim} can be simplified to the following
conditions 
\begin{equation}
\lambda _{\left[ x,y\right] }=0  \label{id3}
\end{equation}
\begin{equation}
\rho _{\left( x\cdot y\right) }=\rho _{y}\circ \rho _{x}  \label{id4}
\end{equation}
\begin{equation*}
\begin{tabular}{l}
$g\left( x,y\cdot z\right) -g\left( y,x\cdot z\right) +\lambda _{x}\left(
g\left( y,z\right) \right) -\lambda _{y}\left( g\left( x,z\right) \right)
-g\left( \left[ x,y\right] ,z\right) $ \\ 
$\quad \quad \quad \quad -\rho _{z}\left( g\left( x,y\right) -g\left(
y,x\right) \right) =0$%
\end{tabular}%
\end{equation*}

By using (\ref{(3)}) and (\ref{(4)}), we deduce from 
\begin{equation*}
\left[ \left( x,a\right) ,\left( y,b\right) \right] =\left( x,a\right) \cdot
\left( y,b\right) -\left( y,b\right) \cdot \left( x,a\right) ,
\end{equation*}
that 
\begin{equation*}
\omega \left( x,y\right) =g(x,y)-g(y,x)
\end{equation*}

From condition (\ref{id3}), we get $\lambda _{e_{2}}=0.$ Applying the
identity (\ref{id4}) above to $e_{i}\cdot e_{i},$ $i=1,2,$ we deduce that $%
\rho =0$ and hence $\lambda _{e_{1}}=\phi \left( e_{1}\right) .$

Assume first that $D\cong \left( 
\begin{array}{ll}
1 & 0 \\ 
0 & 0%
\end{array}
\right) ,$ that is, $\omega \left( e_{1},e_{2}\right) =0$ and $\phi \left(
e_{1}\right) =0,$ then $\lambda =\rho =0.$ Thus, the extension is
central.\medskip

We know that the classification of the exact central extension of $N_{2}$ by 
$\mathbb{R}_{0}$ is, up to left-symmetric isomorphism, the orbit space of $%
H_{ex}^{2}\left( N_{2},\mathbb{R}_{0}\right) $ under the natural action of $%
G=Aut(\mathbb{R}_{0})\times Aut(N_{2})$ (Proposition \ref{isomo}). So, we
must compute $H_{ex}^{2}\left( N_{2},\mathbb{R}_{0}\right) .$ Since $\mathbb{%
R}_{0}\mathbb{\,}$is a trivial $N_{2}$-bimodule, then 
\begin{eqnarray*}
\delta _{1}h\left( x,y\right) &=&-h\left( x\cdot y\right) , \\
\delta _{2}g\left( x,y,z\right) &=&g\left( x,y\cdot z\right) -g\left(
y,x\cdot z\right) -g\left( \left[ x,y\right] ,z\right) ,
\end{eqnarray*}
where $h\in \mathcal{L}^{1}\left( N_{2},\mathbb{R}\right) $ and $g\in 
\mathcal{L}^{2}\left( N_{2},\mathbb{R}\right) .$ This implies that, with
respect to the basis $e_{1},e_{2}$ of $N_{2},$ $\delta _{1}h$ is of the form 
\begin{equation*}
\delta _{1}h=\left( 
\begin{array}{cc}
0 & h_{12} \\ 
0 & 0%
\end{array}
\right) ,
\end{equation*}
where $h_{12}=-h(e_{2}).$

Observe that if $g$ is a $2-$cocycle (i.e. $\delta _{2}g=0$), then 
\begin{equation*}
g=\left( 
\begin{array}{cc}
g_{11} & 0 \\ 
0 & 0%
\end{array}
\right) ,
\end{equation*}
where $g_{ij}=g(e_{i},e_{j}).$ Hence, $\left[ g\right] \in H^{2}\left( N_{2},%
\mathbb{R}\right) $ can be represented as a matrix with respect to $%
\{e_{1},e_{2}\}$ by 
\begin{equation*}
g=\left( 
\begin{array}{cc}
t & 0 \\ 
0 & 0%
\end{array}
\right) ,t\in \mathbb{R}
\end{equation*}

We determine, in this case, the extended left-symmetric structure on $A_{3}.$
By setting $\widetilde{e}_{i}=\left( e_{i},0\right) $, $i=1,2$ and $%
\widetilde{e}_{3}=\left( 0,1\right) $, and using formula (\ref{l.s p}), we
find 
\begin{equation*}
\widetilde{e}_{1}\cdot \widetilde{e}_{1}=t\widetilde{e}_{3},\qquad 
\widetilde{e}_{1}\cdot \widetilde{e}_{2}=\widetilde{e}_{2}
\end{equation*}
and all other products are zero, $t\in \mathbb{R}$. We denote $\mathcal{G}$
endowed with this structure by $N_{3,t}.$

Recall that the extension 
\begin{equation*}
0\rightarrow \mathbb{R}_{0}\rightarrow A_{3}\rightarrow N_{2}\rightarrow 0
\end{equation*}
is exact (i.e. $i(\mathbb{R}_{0})=C(A_{2})$) if and only if $I_{[g]}=\left\{
0\right\} .$

Let $x=ae_{1}+be_{2}\in I_{[g]}.$ Then computing all the products $x\cdot
e_{i}=e_{i}\cdot x=0,$ we deduce that $x=0,$ that is the extension is exact.

Let $N_{3,t},$ $N_{3,t^{^{\prime }}}$ be two left-symmetric algebras as
above. We know that $N_{3,t}$ is isomorphic to $N_{3,t^{^{\prime }}}$ if and
only if there exists $\left( \alpha ,\eta \right) \in Aut(\mathbb{R}%
_{0})\times Aut(N_{2})=\mathbb{R}^{*}\mathbb{\times }Aut(N_{2})$ such that
for all $x,y\in N_{2},$ we have 
\begin{equation}
g^{^{\prime }}\left( x,y\right) =\alpha g\left( \eta \left( x\right) ,\eta
\left( y\right) \right) .  \label{class2}
\end{equation}

Now, we have to calculate $Aut(N_{2}).$ Let $\eta \in Aut(N_{2})$ so that,
with respect to the basis $e_{1},e_{2}$ of $N_{2}$ with $e_{1}\cdot
e_{2}=e_{2},$%
\begin{equation*}
\eta =\left( 
\begin{array}{cc}
a & b \\ 
c & d%
\end{array}
\right) .
\end{equation*}

Since $\eta (e_{2})=\eta (e_{1}\cdot e_{2})=\eta (e_{1})\cdot \eta (e_{2})$,
then $b=0$ and $d=ad.$ Also $0=\eta (e_{1}\cdot e_{1})=\eta (e_{1})\cdot
\eta (e_{1})$ which implies that $a=0$ or $c=0.$ Since $\det \eta \neq 0,$
then $d\neq 0$ and hence $a=1$ and $c=0.$ This means that 
\begin{equation*}
\eta =\left( 
\begin{array}{cc}
1 & 0 \\ 
0 & d%
\end{array}
\right) ,
\end{equation*}
with $d\neq 0.$ We shall now apply formula (\ref{class2}). For this we
recall first that in the basis $e_{1},e_{2},$ the classes $g$ and $%
g^{^{\prime }}$ corresponding to $N_{3,t}$ and $N_{3,t^{^{\prime }}}$ have,
respectively, the forms 
\begin{equation*}
g=\left( 
\begin{array}{cc}
t & 0 \\ 
0 & 0%
\end{array}
\right) \text{ and }g^{^{\prime }}=\left( 
\begin{array}{cc}
t^{^{\prime }} & 0 \\ 
0 & 0%
\end{array}
\right)
\end{equation*}
From $g^{^{\prime }}(e_{1},e_{1})=\alpha g\left( \eta \left( e_{1}\right)
,\eta \left( e_{1}\right) \right) ,$ we get 
\begin{equation*}
t^{^{\prime }}=\alpha t
\end{equation*}

Hence $N_{3,t}$ and $N_{3,t^{^{\prime }}}$ are isomorphic if and only if $%
t^{^{\prime }}=\alpha t,$ for some $\alpha \in \mathbb{R}^{*}$.

Notice that if $t=0,$ we obtain the complete left-symmetric algebra $N_{3,0}$
described above. If $t\neq 0,$ we obtain, by setting $e_{i}=\widetilde{e}%
_{i},$ $i=1,2,$ and $e_{3}=t\widetilde{e}_{3},$ the complete left-symmetric
algebra 
\begin{equation*}
N_{3,1}=\left\langle e_{1},e_{2},e_{3}:e_{1}\cdot e_{1}=e_{3},e_{1}\cdot
e_{2}=e_{2}\right\rangle
\end{equation*}

Assume now that $D\cong \left( 
\begin{array}{ll}
1 & 0 \\ 
0 & 1%
\end{array}
\right) ,$ that is, $\omega \left( e_{1},e_{2}\right) =0$ and $\phi \left(
e_{1}\right) =1.$ Then $\lambda \left( e_{1}\right) =\phi \left(
e_{1}\right) =1.$ We deduce, in this case, that, in the basis $e_{1},e_{2}$
of $N_{2}$, the class $\left[ g\right] \in H_{\lambda ,\rho }^{2}\left(
N_{2},\mathbb{R}\right) $ of a cocycle $g$ may be represented by a matrix of
the simplified form 
\begin{equation*}
g=\left( 
\begin{array}{cc}
0 & t \\ 
t & 0%
\end{array}
\right)
\end{equation*}

We determine, in this case, the extended complete left-symmetric structure
on $A_{3}.$ By setting $\widetilde{e}_{i}=\left( e_{i},0\right) $, $i=1,2$
and $\widetilde{e}_{3}=\left( 0,1\right) $ and using formula (\ref{l.s p}),
we obtain 
\begin{eqnarray*}
\widetilde{e}_{1}\cdot \widetilde{e}_{2} &=&\widetilde{e}_{2}+t\widetilde{e}%
_{3} \\
\widetilde{e}_{2}\cdot \widetilde{e}_{1} &=&t\widetilde{e}_{3} \\
\widetilde{e}_{1}\cdot \widetilde{e}_{3} &=&\widetilde{e}_{3}
\end{eqnarray*}

We denote this left-symmetric algebra by $B_{3,t}.$ Notice that if $t=0,$ we
obtain the complete left-symmetric algebra $B_{3,0}$ with the non-zero
relations 
\begin{eqnarray*}
e_{1}\cdot e_{2} &=&e_{2}, \\
e_{1}\cdot e_{3} &=&e_{3}.
\end{eqnarray*}

If $t\neq 0,$ we obtain, by setting $e_{i}=\widetilde{e}_{i},$ $i=1,2,$ and $%
e_{3}=t\widetilde{e}_{3},$ the complete left-symmetric algebra $B_{3,1}$
with the non-zero relations 
\begin{eqnarray*}
e_{1}\cdot e_{2} &=&e_{2}+e_{3} \\
e_{2}\cdot e_{1} &=&e_{3} \\
e_{1}\cdot e_{3} &=&e_{3}
\end{eqnarray*}

Assume now that $D\cong \left( 
\begin{array}{ll}
1 & 1 \\ 
0 & 1%
\end{array}
\right) $ that is, $\omega \left( e_{1},e_{2}\right) =1$ and $\phi \left(
e_{1}\right) $ $=1.$ Hence $\lambda \left( e_{1}\right) =\phi \left(
e_{1}\right) =1.$ Using the same method as above, it follows that the class $%
\left[ g\right] \in H_{\lambda ,\rho }^{2}\left( N_{2},\mathbb{R}\right) $
of a cocycle $g$ takes the reduced form 
\begin{equation*}
g=\left( 
\begin{array}{cc}
0 & t \\ 
t-1 & 0%
\end{array}
\right)
\end{equation*}

We determine, in this case, the extended complete left-symmetric structures
on $A_{3}.$ By setting $\widetilde{e}_{i}=\left( e_{i},0\right) $, $i=1,2$
and $\widetilde{e}_{3}=\left( 0,1\right) $ and using formula (\ref{l.s p}),
we obtain 
\begin{eqnarray*}
\widetilde{e}_{1}\cdot \widetilde{e}_{2} &=&\widetilde{e}_{2}+t\widetilde{e}%
_{3} \\
\widetilde{e}_{2}\cdot \widetilde{e}_{1} &=&\left( t-1\right) \widetilde{e}%
_{3} \\
\widetilde{e}_{1}\cdot \widetilde{e}_{3} &=&\widetilde{e}_{3}
\end{eqnarray*}

We denote such a left-symmetric algebra by $C_{3,t}.$ Notice that if $t=1,$
we obtain the complete left-symmetric algebra $C_{3,1}$ with the non-zero
relations 
\begin{eqnarray*}
e_{1}\cdot e_{2} &=&e_{2}+e_{3}, \\
e_{1}\cdot e_{3} &=&e_{3},
\end{eqnarray*}
and if $t\neq 1,$ we obtain the complete left-symmetric algebra $C_{3,t}$
with the non-zero relations 
\begin{eqnarray*}
e_{1}\cdot e_{2} &=&e_{2}+te_{3} \\
e_{2}\cdot e_{1} &=&\left( t-1\right) e_{3} \\
e_{1}\cdot e_{3} &=&e_{3}
\end{eqnarray*}
where different values of $t$ give non-isomorphic complete left-symmetric
algebras.

Assume finally that $D\cong \left( 
\begin{array}{ll}
1 & 0 \\ 
0 & \mu%
\end{array}
\right) ,$ with $0<\left| \mu \right| <1,$ that is $\omega \left(
e_{1},e_{2}\right) =0$ and $\phi \left( e_{1}\right) =\mu .$ Hence $\lambda
\left( e_{1}\right) =\phi \left( e_{1}\right) =\mu .$ It follows that the
class $\left[ g\right] \in H_{\lambda ,\rho }^{2}\left( N_{2},\mathbb{R}%
\right) $ of a cocycle $g$ is identically zero.

We determine, in this case, the extended complete left-symmetric structures
on $A_{3}.$ By setting $\widetilde{e}_{i}=\left( e_{i},0\right) $, $i=1,2$
and $\widetilde{e}_{3}=\left( 0,1\right) $ and using formula (\ref{l.s p}),
we obtain 
\begin{eqnarray*}
\widetilde{e}_{1}\cdot \widetilde{e}_{2} &=&\widetilde{e}_{2}. \\
\widetilde{e}_{1}\cdot \widetilde{e}_{3} &=&\mu \widetilde{e}_{3}.
\end{eqnarray*}
where $0<\left| \mu \right| <1.$ We set 
\begin{equation*}
D_{3,1}\left( \mu \right) =\left\langle e_{1},e_{2},e_{3:}e_{1}\cdot
e_{2}=e_{2},e_{1}\cdot e_{3}=\mu e_{3}\right\rangle
\end{equation*}
where $0<\left| \mu \right| <1.$

\begin{itemize}
\item Case 2. $\mathcal{I}$ $\cong $aff$\left( \mathbb{R}\right) .$
\end{itemize}

In this case, the short exact sequence (\ref{(1)}) becomes 
\begin{equation}
0\rightarrow N_{2}\rightarrow A_{3}\rightarrow \mathbb{R}_{0}\rightarrow 0
\label{ses1}
\end{equation}
where $N_{2}$ is the complete left-symmetric algebra whose associated Lie
algebra is aff$\left( \mathbb{R}\right) $ and $\mathbb{R}_{0}$ is the
trivial left-symmetric algebra over $\mathbb{R}$.

Let $\sigma :\mathbb{R}_{0}\rightarrow A_{3}$ be a section and set $\sigma
\left( 1\right) =x_{\circ }\in A_{3}$ and define two linear maps $\lambda
,\rho \in End\left( N_{2}\right) $ by putting $\lambda \left( y\right)
=x_{\circ }\cdot y$ and $\rho \left( y\right) =y\cdot x_{\circ }.$ By
setting $e=x_{\circ }\cdot x_{\circ }$, we see that $e\in N_{2}.$ Let $g:%
\mathbb{R}_{0}\mathbb{\times R}_{0}\rightarrow N_{2}$ be the bilinear map
defined by $g\left( a,b\right) =\sigma \left( a\right) \cdot \sigma \left(
b\right) -\sigma \left( a\cdot b\right) .$ Since the complete left-symmetric
structure on $\mathbb{R}$ is trivial, then $g\left( a,b\right) =abe,$ or
equivalently $g\left( 1,1\right) =e.$ Also we can show that $\delta _{2}g=0,$
i.e., $g\in Z_{\lambda ,\rho }^{2}\left( \mathbb{R}_{0},N_{2}\right) .$

In this case, the extended left-symmetric product on $\mathbb{R}_{0}\oplus
N_{2} $ given by (\ref{product}) takes the simplified form 
\begin{equation*}
\left( a,x\right) \cdot \left( b,y\right) =\left( 0,x\cdot y+a\lambda \left(
y\right) +b\rho \left( x\right) +abe\right) ,
\end{equation*}
for all $a,b\in \mathbb{R}$ and $x,y\in N_{2}.$

The conditions in Theorem \ref{kim} can be simplified to the following
conditions 
\begin{equation}
\lambda \left( x\cdot y\right) =\lambda \left( x\right) \cdot y+x\cdot
\lambda \left( y\right) -\rho \left( x\right) \cdot y  \label{id5}
\end{equation}
\begin{equation}
\rho \left( \left[ x,y\right] \right) =x\cdot \rho \left( y\right) -y\cdot
\rho \left( x\right)  \label{id6}
\end{equation}
\begin{equation}
\left[ \lambda ,\rho \right] +\rho ^{2}=R_{e}  \label{id7}
\end{equation}

Let $\phi :\mathbb{R}\rightarrow Der\left( \text{aff}\left( \mathbb{R}%
\right) \right) ,$ be a derivation of aff$\left( \mathbb{R}\right) $. Set 
\begin{equation*}
\phi \left( 1\right) =\left( 
\begin{array}{cc}
a & c \\ 
b & d%
\end{array}
\right)
\end{equation*}
relative to a basis $e_{1},e_{2}$ of aff$\left( \mathbb{R}\right) $
satisfying $\left[ e_{1},e_{2}\right] =e_{2}.$ From the identity $\phi
\left( 1\right) e_{2}=\left[ \phi \left( 1\right) e_{1},e_{2}\right] +\left[
e_{1},\phi \left( 1\right) e_{2}\right] ,$ we deduce that $a=c=0$, hence 
\begin{equation*}
\phi \left( 1\right) =\left( 
\begin{array}{cc}
0 & 0 \\ 
b & d%
\end{array}
\right)
\end{equation*}

Let 
\begin{equation*}
\rho =\left( 
\begin{array}{cc}
\alpha _{1} & \beta _{1} \\ 
\alpha _{2} & \beta _{2}%
\end{array}
\right)
\end{equation*}
relative to a basis $e_{1},e_{2}$ of aff$\left( \mathbb{R}\right) $
satisfying $\left[ e_{1},e_{2}\right] =e_{2}.$ Applying formula (\ref{id6})
to $e_{2},$ we get $\beta _{1}=0.$ Since $\phi \left( 1\right) =\lambda
-\rho ,$ we deduce that, relative to the basis $e_{1},e_{2}$, we have 
\begin{equation*}
\lambda =\left( 
\begin{array}{cc}
\alpha _{1} & 0 \\ 
\alpha _{2}+b & \beta _{2}+d%
\end{array}
\right)
\end{equation*}

Applying formula (\ref{id5}) to all products of the form $e_{i}\cdot e_{j},$ 
$i=1,2,$ we get $\alpha _{2}+b=0.$ Moreover, by applying formula (\ref{id7})
to $e_{1}$ and $e_{2},$ we get $\alpha _{1}=\beta _{2}=0$. Thus 
\begin{equation*}
\rho =\left( 
\begin{array}{cc}
0 & 0 \\ 
-b & 0%
\end{array}
\right) \text{ and }\lambda =\left( 
\begin{array}{cc}
0 & 0 \\ 
0 & d%
\end{array}
\right)
\end{equation*}

Now, since $e\in N_{2},$ then $e=te_{1}+se_{2}$ for some $t,s\in \mathbb{R}$%
. Formula (\ref{id7}) when applied to $e_{1}$ gives 
\begin{equation*}
-bde_{2}=se_{2}
\end{equation*}
for which we get that $e=x_{\circ }\cdot x_{\circ }=te_{1}-bde_{2},$ $t\in 
\mathbb{\dot{R}}$. Hence we get a left-symmetric product on $A_{3}$.

Now, let us write down the structure of $A_{3}$ using a basis. From above we
have 
\begin{eqnarray*}
e_{1}\cdot e_{2} &=&e_{2},\qquad e_{1}\cdot x_{\circ }=-be_{2} \\
x_{\circ }\cdot e_{2} &=&de_{2},\qquad x_{\circ }\cdot x_{\circ
}=te_{1}-bde_{2},\text{ }t\in \mathbb{R}
\end{eqnarray*}

Since $x_{0}\in A_{3}$ and $\pi \left( x_{0}\right) =1,$ then $x_{0}\in
A_{3}\setminus N_{2}.$ Indeed if $x_{0}\in N_{2},$ then the exactness of the
short sequence (\ref{ses1}) implies that $x_{0}\in i\left( N_{2}\right)
=\ker \pi ,$ a contradiction. This implies that, relative to a basis $%
\left\{ e_{1},e_{2},e_{3}\right\} $ of $A_{3},$ $x_{0}$ is of the form $%
x_{0}=\alpha e_{1}+\beta e_{2}+\gamma e_{3},$ where $\alpha ,\beta ,\gamma
\in \mathbb{R}$ with $\gamma \neq 0$. In this case, we can, without loss of
generality, assume that $\gamma =1$. Thus, $e_{3}=x_{0}-\alpha e_{1}-\beta
e_{2}.$ Since $e_{1}\cdot x_{\circ }=-be_{2}$ we get that 
\begin{equation*}
e_{1}\cdot e_{3}=-\left( b+\beta \right) e_{2},
\end{equation*}
also since $x_{\circ }\cdot e_{2}=de_{2}$ we get 
\begin{equation*}
e_{3}\cdot e_{2}=\left( d-\alpha \right) e_{2}.
\end{equation*}
Since $x_{\circ }\cdot x_{\circ }=te_{1}-bde_{2,}$ we deduce that 
\begin{equation*}
e_{3}\cdot e_{3}=te_{1}+\left( \alpha b+\alpha \beta -bd-\beta d\right)
e_{2}.
\end{equation*}

Since $\alpha ,\beta $ are arbitrary, we can choose $\alpha ,\beta $ so that 
$e_{3}=x_{\circ }-de_{1}-be_{2}$. Hence the left-symmetric product on $A_{3}$
is given, relative the basis $\left\{ e_{1},e_{2},e_{3}\right\} ,$ by the
non-zero relations 
\begin{eqnarray*}
e_{1}\cdot e_{2} &=&e_{2} \\
e_{3}\cdot e_{3} &=&te_{1},
\end{eqnarray*}

Notice that if $t=0,$ we obtain the complete left-symmetric algebra $N_{3,0}$%
. If $t\neq 0,$ we obtain, by setting $\widetilde{e}_{i}=e_{i},$ $i=1,2$ and 
$\widetilde{e}_{3}=\frac{1}{\sqrt{\left| t\right| }}e_{3},$ that $A_{3}$ is
isomorphic to one of the left-symmetric algebras $N_{3,2}$ or $N_{3,3}$
given above.

\begin{itemize}
\item Case 3.$\mathcal{I}$ $\cong \mathbb{R}^{2}.$
\end{itemize}

In this case, the short exact sequence (\ref{(1)}) becomes 
\begin{equation*}
0\rightarrow A_{2}\rightarrow A_{3}\rightarrow \mathbb{R}_{0}\rightarrow 0
\end{equation*}
where $A_{2}$ is a complete left-symmetric algebra whose Lie algebra is $%
\mathbb{R}^{2}$ and $\mathbb{R}_{0}$ is the trivial left-symmetric algebra
over $\mathbb{R}$.

At the Lie algebra level, we have a short exact sequence of Lie algebras of
the form 
\begin{equation*}
0\rightarrow \mathbb{R}^{2}\rightarrow \widetilde{\mathcal{G}}\rightarrow 
\mathbb{R}\rightarrow 0
\end{equation*}

Let $\phi :\mathbb{R}\rightarrow Der\left( \mathbb{R}^{2}\right) \cong
End\left( \mathbb{R}^{2}\right) ,$ be a derivation of $\mathbb{R}^{2}$.
Relative to a basis $e_{1},e_{2}$ of $\mathbb{R}^{2},$ set 
\begin{equation*}
\phi \left( 1\right) =\left( 
\begin{array}{cc}
a & c \\ 
b & d%
\end{array}
\right)
\end{equation*}

In this case, the extended Lie bracket on $\mathbb{R}\times \mathbb{R}^{2},$
given by (\ref{ext Lie bracket}), takes the simplified form 
\begin{equation*}
\left[ \left( a,x\right) ,\left( b,y\right) \right] =\left( 0,\phi \left(
a\right) y-\phi \left( b\right) x+\omega \left( a,b\right) \right) ,
\end{equation*}
for all $x,y\in \mathbb{R}^{2}$ and $a,b\in \mathbb{R}$. By setting $%
\widetilde{e}_{1}=\left( 1,0\right) $ and $\widetilde{e}_{i+1}=\left(
0,e_{i}\right) $, $i=1,2$ we obtain 
\begin{eqnarray*}
\left[ \widetilde{e}_{1},\widetilde{e}_{2}\right] &=&a\widetilde{e}_{1}+b%
\widetilde{e}_{2} \\
\left[ \widetilde{e}_{1},\widetilde{e}_{3}\right] &=&c\widetilde{e}_{1}+d%
\widetilde{e}_{2} \\
\left[ \widetilde{e}_{2},\widetilde{e}_{3}\right] &=&0
\end{eqnarray*}

By Lemma \ref{milnorj}, we obtain that, relative to the basis $e_{1},e_{2},$ 
\begin{equation*}
D=\left( 
\begin{array}{ll}
a & b \\ 
c & d%
\end{array}
\right)
\end{equation*}
with $a+d\neq 0.$ Note that, in this case, $\omega $ may not be zero, that
is, the extensions of $\mathbb{R}$ by $\mathbb{R}^{2}$ are not necessarily
semidirect products of $\mathbb{R}$ by $\mathbb{R}^{2}.$

According to Lemma \ref{milnorj}, there are five cases to be considered 
\begin{equation*}
D\cong \left( 
\begin{array}{ll}
1 & 0 \\ 
0 & 0%
\end{array}
\right) ,\text{ }\left( 
\begin{array}{ll}
1 & 0 \\ 
0 & 1%
\end{array}
\right) ,\text{ }\left( 
\begin{array}{ll}
1 & 1 \\ 
0 & 1%
\end{array}
\right) ,\text{ }\left( 
\begin{array}{ll}
1 & 0 \\ 
0 & \mu%
\end{array}
\right) \text{ or }\left( 
\begin{array}{ll}
1 & -\zeta \\ 
\zeta & 1%
\end{array}
\right) ,
\end{equation*}
where $\zeta >0$ and $0<\left| \mu \right| <1.$

Let $\sigma :\mathbb{R}_{0}\rightarrow A_{3}$ be a section and set $\sigma
\left( 1\right) =x_{\circ }\in A_{3}$ and define two linear maps $\lambda
,\rho \in End\left( A_{2}\right) $ by putting $\lambda \left( y\right)
=x_{\circ }\cdot y$ and $\rho \left( y\right) =y\cdot x_{\circ }.$ By
setting $e=x_{\circ }\cdot x_{\circ },$ we see that $e\in A_{2}.$ Let $g:%
\mathbb{R}_{0}\mathbb{\times R}_{0}\rightarrow A_{2}$ be the bilinear map
defined by $g\left( a,b\right) =\sigma \left( a\right) \cdot \sigma \left(
b\right) -\sigma \left( a\cdot b\right) .$ Since the complete left-symmetric
structure on $\mathbb{R}$ is trivial, then $g\left( a,b\right) =abe,$ or
equivalently $g\left( 1,1\right) =e.$ Also we can show that $\delta _{2}g=0,$
i.e., $g\in Z_{\lambda ,\rho }^{2}\left( \mathbb{R}_{0},A_{2}\right) .$

The extended left-symmetric product on $\mathbb{R}_{0}\oplus A_{2}$ given by
(\ref{product}) is then takes the simplified form 
\begin{equation}
\left( a,x\right) \cdot \left( b,y\right) =\left( 0,x\cdot y+a\lambda \left(
y\right) +b\rho \left( x\right) +abe\right)  \label{lsf}
\end{equation}
for all $x,y\in A_{2}$ and $a,b\in \mathbb{R}$.

The conditions in Theorem \ref{kim} can be simplified to the following
conditions 
\begin{equation}
\lambda \left( x\cdot y\right) =\lambda \left( x\right) \cdot y+x\cdot
\lambda \left( y\right) -\rho \left( x\right) \cdot y  \label{id8}
\end{equation}
\begin{equation}
x\cdot \rho \left( y\right) -y\cdot \rho \left( x\right) =0  \label{id9}
\end{equation}
\begin{equation}
\left[ \lambda ,\rho \right] +\rho ^{2}=R_{e}  \label{id10}
\end{equation}

According to Lemma \ref{r2}, we have the following cases of $A_{2}$

\begin{enumerate}
\item $A_{2}=\left\langle e_{1},e_{2}:e_{i}\cdot
e_{j}=0,i,j=1,2\right\rangle .$

Assume first that $D\cong \left( 
\begin{array}{ll}
1 & 0 \\ 
0 & 0%
\end{array}
\right) $ and let 
\begin{equation*}
\rho =\left( 
\begin{array}{cc}
\alpha _{1} & \beta _{1} \\ 
\alpha _{2} & \beta _{2}%
\end{array}
\right)
\end{equation*}
relative to the basis $e_{1},e_{2}$ of $A_{2}.$ Since $\phi \left( 1\right)
=\lambda -\rho ,$ we deduce that, relative to the basis $e_{1},e_{2}$, we
have 
\begin{equation*}
\lambda =\left( 
\begin{array}{cc}
\alpha _{1}+1 & \beta _{1} \\ 
\alpha _{2} & \beta _{2}%
\end{array}
\right)
\end{equation*}

Applying formula (\ref{id10}) to $e_{2},$ we obtain $\beta _{1}=\beta
_{2}=0. $ The same formula when applied to $e_{1}$ yields $\alpha
_{1}=\alpha _{2}=0. $ It follows that $\rho $ is identically zero and 
\begin{equation*}
\lambda =\left( 
\begin{array}{cc}
1 & 0 \\ 
0 & 0%
\end{array}
\right)
\end{equation*}

We can easily show that the condition (\ref{id10}) above is satisfied for
all $e=x_{\circ }\cdot x_{\circ }=se_{1}+te_{2},$ $s,t\in \mathbb{R}$. Hence
we get a left-symmetric product on $A_{3}$.

Now, let us write down the structure of $A_{3}$ using a basis. From above we
have 
\begin{equation*}
x_{\circ }\cdot e_{1}=e_{1},\qquad x_{\circ }\cdot x_{\circ }=se_{1}+te_{2}.
\end{equation*}

We can easily prove that $x_{0}\in A_{3}\setminus A_{2}$. This implies that,
relative to a basis $\left\{ e_{1},e_{2},e_{3}\right\} $ of $A_{3},$ $x_{0}$
is of the form $x_{0}=\alpha e_{1}+\beta e_{2}+\gamma e_{3},$ where $\alpha
,\beta ,\gamma \in \mathbb{R}$ with $\gamma \neq 0$. In this case, we can,
without loss of generality, assume that $\gamma =1$. Thus, $%
e_{3}=x_{0}-\alpha e_{1}-\beta e_{2}.$ Since $x_{\circ }\cdot e_{1}=e_{1}$
we get that 
\begin{equation*}
e_{3}\cdot e_{1}=e_{1}
\end{equation*}
also since $x_{\circ }\cdot x_{\circ }=se_{1}+te_{2,}$ we deduce that 
\begin{equation*}
e_{3}\cdot e_{3}=\left( s-\alpha \right) e_{1}+te_{2}.
\end{equation*}

Since $\alpha ,\beta $ are arbitrary, we can choose $\alpha ,\beta $ so that 
$e_{3}=x_{\circ }-se_{1}$. Hence the left-symmetric product on $A_{3}$ is
given, relative to the basis $\left\{ e_{1},e_{2},e_{3}\right\} $ of $A_{3},$
by the non-zero relations 
\begin{eqnarray*}
e_{3}\cdot e_{1} &=&e_{1} \\
e_{3}\cdot e_{3} &=&te_{2}
\end{eqnarray*}

Notice that if $t=0,$ we find the complete left-symmetric algebra $N_{3,0}$.
If $t\neq 0,$ we get, by setting $\widetilde{e}_{1}=e_{3},$ $\widetilde{e}%
_{2}=e_{1}$ and $\widetilde{e}_{3}=te_{2},$ that $A_{3}$ is isomorphic to
the complete left-symmetric algebra $N_{3,1}.$

Assume then that $D\cong \left( 
\begin{array}{ll}
1 & 0 \\ 
0 & 1%
\end{array}
\right) $ and let 
\begin{equation*}
\rho =\left( 
\begin{array}{cc}
\alpha _{1} & \beta _{1} \\ 
\alpha _{2} & \beta _{2}%
\end{array}
\right) ,
\end{equation*}
relative to the basis $e_{1},e_{2}$ of $A_{2}.$ Since $\phi \left( 1\right)
=\lambda -\rho ,$ we deduce that, relative to the basis $e_{1},e_{2}$, we
have 
\begin{equation*}
\lambda =\left( 
\begin{array}{cc}
\alpha _{1}+1 & \beta _{1} \\ 
\alpha _{2} & \beta _{2}+1%
\end{array}
\right) .
\end{equation*}

By applying formula (\ref{id10}) to $e_{1}$ and $e_{2},$ we get 
\begin{equation*}
\rho =\left( 
\begin{array}{cc}
0 & \alpha \\ 
0 & 0%
\end{array}
\right) ,\text{ }\lambda =\left( 
\begin{array}{cc}
1 & \alpha \\ 
0 & 1%
\end{array}
\right) ,\text{ }\alpha \in \mathbb{R}
\end{equation*}
and $e=x_{\circ }\cdot x_{\circ }=\alpha ^{2}e_{1}+\alpha e_{2}.$

Similarly, we find that, relative to the basis $\left\{
e_{1},e_{2},e_{3}\right\} $ of $A_{3}$ with $e_{3}=x_{\circ }+\alpha
^{2}e_{1}-\alpha e_{2}$, the left-symmetric product on $A_{3}$ is given by
the non-zero relations 
\begin{eqnarray*}
e_{3}\cdot e_{1} &=&e_{1} \\
e_{3}\cdot e_{2} &=&\alpha e_{1}+e_{2} \\
e_{2}\cdot e_{3} &=&\alpha e_{1}.
\end{eqnarray*}

Notice that if $\alpha =0,$ we get, by setting $\widetilde{e}_{1}=e_{3},$ $%
\widetilde{e}_{2}=e_{1}$ and $\widetilde{e}_{3}=e_{2},$ the complete
left-symmetric algebra $B_{3,0}$. If $t\neq 0,$ we get, by setting $%
\widetilde{e}_{1}=e_{3},$ $\widetilde{e}_{2}=e_{2}$ and $\widetilde{e}%
_{3}=\alpha e_{1},$ that $A_{3}$ is isomorphic to the complete
left-symmetric algebras $B_{3,1}.$

Assume now that $D\cong \left( 
\begin{array}{ll}
1 & 1 \\ 
0 & 1%
\end{array}
\right) ,$ and let 
\begin{equation*}
\rho =\left( 
\begin{array}{cc}
\alpha _{1} & \beta _{1} \\ 
\alpha _{2} & \beta _{2}%
\end{array}
\right) ,
\end{equation*}
relative to the basis $e_{1},e_{2}$ of $A_{2}.$ Since $D=\lambda -\rho ,$ we
deduce that, relative to the basis $e_{1},e_{2}$, we have 
\begin{equation*}
\lambda =\left( 
\begin{array}{cc}
\alpha _{1}+1 & \beta _{1}+1 \\ 
\alpha _{2} & \beta _{2}+1%
\end{array}
\right) .
\end{equation*}

By applying formula (\ref{id10}) to $e_{1}$ and $e_{2},$ we get 
\begin{equation*}
\rho =\left( 
\begin{array}{cc}
0 & \alpha \\ 
0 & 0%
\end{array}
\right) \text{, }\lambda =\left( 
\begin{array}{cc}
1 & \alpha +1 \\ 
0 & 1%
\end{array}
\right) ,\text{ }\alpha \in \mathbb{R}
\end{equation*}
and $e=x_{\circ }\cdot x_{\circ }=\alpha e_{1}+\alpha e_{2}.$

Similarly, we find that, relative to a basis $\left\{
e_{1},e_{2},e_{3}\right\} $ of $A_{3}$ with $e_{3}=x_{\circ }+2\alpha
^{2}e_{1}-\alpha e_{2}$, the left-symmetric product on $A_{3}$ is given by
the non-zero relations 
\begin{eqnarray*}
e_{3}\cdot e_{1} &=&e_{1} \\
e_{3}\cdot e_{2} &=&\left( \alpha +1\right) e_{1}+e_{2} \\
e_{2}\cdot e_{3} &=&\alpha e_{1}.
\end{eqnarray*}

Notice that if $\alpha =0,$ we get, by setting $\widetilde{e}_{1}=e_{3},$ $%
\widetilde{e}_{2}=e_{2}$ and $\widetilde{e}_{3}=e_{1},$ the complete
left-symmetric algebra $C_{3,1}$. If $\alpha \neq 0,$ we get, by setting $%
\alpha =t-1$ with $t\neq 1,$ the complete left-symmetric algebra $C_{3,t}$
where different values of $t$ give non-isomorphic complete left-symmetric
algebras.

Assume then that $D\cong \left( 
\begin{array}{ll}
1 & 0 \\ 
0 & \mu%
\end{array}
\right) ,$ where $0<\left| \mu \right| <1,$ and let 
\begin{equation*}
\rho =\left( 
\begin{array}{cc}
\alpha _{1} & \beta _{1} \\ 
\alpha _{2} & \beta _{2}%
\end{array}
\right) ,
\end{equation*}
relative to the basis $e_{1},e_{2}$ of $A_{2}.$ Since $\phi \left( 1\right)
=\lambda -\rho ,$ we deduce that, relative to the basis $e_{1},e_{2}$, we
have 
\begin{equation*}
\lambda =\left( 
\begin{array}{cc}
\alpha _{1}+1 & \beta _{1} \\ 
\alpha _{2} & \beta _{2}+\mu%
\end{array}
\right) .
\end{equation*}

By applying formula (\ref{id10}) to $e_{1}$ and $e_{2},$ we obtain that $%
\rho $ is identically zero, 
\begin{equation*}
\lambda =\left( 
\begin{array}{cc}
1 & 0 \\ 
0 & \mu%
\end{array}
\right)
\end{equation*}
and $e=x_{\circ }\cdot x_{\circ }=e_{1}+\mu e_{2}.$

Similarly, we find that, relative to a basis $\left\{
e_{1},e_{2},e_{3}\right\} $ of $A_{3}$ with $e_{3}=x_{\circ }-e_{1}-e_{2}$,
the left-symmetric product on $A_{3}$ is given by the non-zero relations 
\begin{eqnarray*}
e_{3}\cdot e_{1} &=&e_{1} \\
e_{3}\cdot e_{2} &=&\mu e_{2}.
\end{eqnarray*}

By setting $\widetilde{e}_{1}=e_{3},$ $\widetilde{e}_{2}=e_{1}$ and $%
\widetilde{e}_{3}=e_{2},$ we get the complete left-symmetric algebra $%
D_{3,1}\left( \mu \right) $ where $0<\left| \mu \right| <1.$

Assume finally that $D\cong \left( 
\begin{array}{ll}
1 & -\zeta \\ 
\zeta & 1%
\end{array}
\right) ,$ where $\zeta >0,$ and let 
\begin{equation*}
\rho =\left( 
\begin{array}{cc}
\alpha _{1} & \beta _{1} \\ 
\alpha _{2} & \beta _{2}%
\end{array}
\right)
\end{equation*}
relative to the basis $e_{1},e_{2}$ of $A_{2}.$ Since $\phi \left( 1\right)
=\lambda -\rho ,$ we deduce that, relative to the basis $e_{1},e_{2}$ above,
we have 
\begin{equation*}
\lambda =\left( 
\begin{array}{cc}
\alpha _{1}+1 & \beta _{1}-\zeta \\ 
\alpha _{2}+\zeta & \beta _{2}+1%
\end{array}
\right)
\end{equation*}

By applying formula (\ref{id10}) to $e_{1}$ and $e_{2},$ we obtain that $%
\rho $ is identically zero, 
\begin{equation*}
\lambda =\left( 
\begin{array}{cc}
1 & -\zeta \\ 
\zeta & 1%
\end{array}
\right)
\end{equation*}
and $e=x_{\circ }\cdot x_{\circ }=2\zeta e_{1}+\left( \zeta ^{2}-1\right)
e_{2}.$

Similarly, we find that, relative to a basis $\left\{
e_{1},e_{2},e_{3}\right\} $ of $A_{3}$ with $e_{3}=x_{\circ }-\zeta
e_{1}+e_{2}$, the left-symmetric product on $A_{3}$ is given by the non-zero
relations 
\begin{eqnarray*}
e_{3}\cdot e_{1} &=&e_{1}+\zeta e_{2} \\
e_{3}\cdot e_{2} &=&-\zeta e_{1}+e_{2}.
\end{eqnarray*}

Set $\widetilde{e}_{1}=e_{3},$ $\widetilde{e}_{2}=e_{1}$ and $\widetilde{e}%
_{3}=e_{2}.$ Then, the non-zero relations above become 
\begin{eqnarray*}
\widetilde{e}_{1}\cdot \widetilde{e}_{2} &=&\widetilde{e}_{2}+\zeta 
\widetilde{e}_{3}, \\
\widetilde{e}_{1}\cdot \widetilde{e}_{3} &=&-\zeta \widetilde{e}_{2}+%
\widetilde{e}_{3}.
\end{eqnarray*}

We set 
\begin{equation*}
E_{3,\zeta }=\left\langle e_{1},e_{2},e_{3}:e_{1}\cdot e_{2}=e_{2}+\zeta
e_{3},e_{1}\cdot e_{3}=-\zeta e_{2}+e_{3},\zeta >0\right\rangle .
\end{equation*}

\item $A_{2}=\left\langle e_{1},e_{2}:e_{2}\cdot e_{2}=e_{1}\right\rangle .$

Let 
\begin{equation*}
\rho =\left( 
\begin{array}{cc}
\alpha _{1} & \beta _{1} \\ 
\alpha _{2} & \beta _{2}%
\end{array}
\right)
\end{equation*}
relative to the basis $e_{1},e_{2}$ of $A_{2}.$ By applying formula (\ref%
{id9}) to $e_{1}$ and $e_{2},$ we get that $\alpha _{2}=0.$

Assume first that $D\cong \left( 
\begin{array}{ll}
1 & 0 \\ 
0 & 0%
\end{array}
\right) .$ Then, as $\phi \left( 1\right) =\lambda -\rho ,$ we deduce that,
relative to the basis $e_{1},e_{2}$, we have 
\begin{equation*}
\lambda =\left( 
\begin{array}{cc}
\alpha _{1}+1 & \beta _{1} \\ 
0 & \beta _{2}%
\end{array}
\right)
\end{equation*}

By applying formula (\ref{id10}) to $e_{1}$ and $e_{2},$ we get that $\alpha
_{1}=\beta _{2}=0.$ Moreover, by applying formula (\ref{id8}) to all
products of the form $e_{i}\cdot e_{j},$ $i,j=1,2$, we get that $1=0,$ a
contradiction. Thus $D$ can not be of this form. Similarly, we can prove
that $D$ can not be of the forms $\left( 
\begin{array}{ll}
1 & 0 \\ 
0 & 1%
\end{array}
\right) ,$ $\left( 
\begin{array}{ll}
1 & 1 \\ 
0 & 1%
\end{array}
\right) ,$ or $\left( 
\begin{array}{ll}
1 & -\zeta \\ 
\zeta & 1%
\end{array}
\right) ,$ where $\zeta >0.$

Assume that $D\cong \left( 
\begin{array}{ll}
1 & 0 \\ 
0 & \mu%
\end{array}
\right) ,$ where $0<\left| \mu \right| <1,$ Then, as $\phi \left( 1\right)
=\lambda -\rho ,$ we deduce that 
\begin{equation*}
\lambda =\left( 
\begin{array}{cc}
\alpha _{1}+1 & \beta _{1} \\ 
0 & \beta _{2}+\mu%
\end{array}
\right)
\end{equation*}

By applying formula (\ref{id10}) to $e_{1}$ and $e_{2},$ we get that $\alpha
_{1}=\beta _{2}=0.$ Moreover, by applying formula (\ref{id8}) to all
products of the form $e_{i}\cdot e_{j},$ $i,j=1,2,$ we get that $\mu =\frac{1%
}{2}.$ Thus 
\begin{equation*}
\rho =\left( 
\begin{array}{cc}
0 & \alpha \\ 
0 & 0%
\end{array}
\right) \text{, }\lambda =\left( 
\begin{array}{ll}
1 & \alpha \\ 
0 & \frac{1}{2}%
\end{array}
\right) ,\text{ }\alpha \in \mathbb{R}
\end{equation*}
and $e=x_{\circ }\cdot x_{\circ }=te_{1}+\frac{1}{2}\alpha e_{2},$ $t\in 
\mathbb{R}$.

Similarly, we find that, relative to a basis $\left\{
e_{1},e_{2},e_{3}\right\} $ of $A_{3}$ with $e_{3}=x_{\circ }+\left( \alpha
^{2}-t\right) e_{1}-\alpha e_{2}$, the left-symmetric product on $A_{3}$ is
given by the non-zero relations 
\begin{eqnarray*}
e_{2}\cdot e_{2} &=&e_{1}, \\
e_{3}\cdot e_{1} &=&e_{1}, \\
e_{3}\cdot e_{2} &=&\frac{1}{2}e_{2},
\end{eqnarray*}
\end{enumerate}

Set $\widetilde{e}_{1}=e_{3},$ $\widetilde{e}_{2}=e_{1}$ and $\widetilde{e}%
_{3}=e_{2}.$ Then the non-zero relations above become 
\begin{eqnarray*}
\widetilde{e}_{2}\cdot \widetilde{e}_{2} &=&\widetilde{e}_{1}, \\
\widetilde{e}_{1}\cdot \widetilde{e}_{2} &=&\widetilde{e}_{2}, \\
\widetilde{e}_{1}\cdot \widetilde{e}_{3} &=&\frac{1}{2}\widetilde{e}_{3}
\end{eqnarray*}
We set 
\begin{equation*}
D_{3,2}=\left\langle e_{1},e_{2},e_{3}:e_{2}\cdot e_{2}=e_{1},e_{1}\cdot
e_{2}=e_{2},e_{1}\cdot e_{3}=\frac{1}{2}e_{3}\right\rangle .
\end{equation*}

\subsection{The classification}

We can now state the main result of this paper

\begin{theorem}
\label{complete3}\label{th1}Let $A_{3}$ be a three dimensional complete
left-symmetric algebra whose associated Lie algebra $\mathcal{G}$ is
solvable and non-unimodular. Then $A_{3}$ is isomorphic to one of the
following left-symmetric algebras:

\begin{equation*}
\begin{tabular}{|c|c|c|c|}
\hline
Name & Non-zero product & Lie algebra & Remarks \\ \hline
$N_{3,0}$ & $e_{1}\cdot e_{2}=e_{2}$ & $\mathcal{G}_{3,1}$ & N, D, S \\ 
\hline
$N_{3,1}$ & $e_{1}\cdot e_{1}=e_{3},$ $e_{1}\cdot e_{2}=e_{2}$ & $\mathcal{G}%
_{3,1}$ & N, D, S \\ \hline
$N_{3,2}$ & $e_{1}\cdot e_{2}=e_{2},$ $e_{3}\cdot e_{3}=e_{1}$ & $\mathcal{G}%
_{3,1}$ & S \\ \hline
$N_{3,3}$ & $e_{1}\cdot e_{2}=e_{2},$ $e_{3}\cdot e_{3}=-e_{1}$ & $\mathcal{G%
}_{3,1}$ & S \\ \hline
$B_{3,0}$ & $e_{1}\cdot e_{2}=e_{2},$ $e_{1}\cdot e_{3}=e_{3}$ & $\mathcal{G}%
_{3,2}$ & N, D, S \\ \hline
$B_{3,1}$ & $\left. 
\begin{array}{c}
e_{1}\cdot e_{2}=e_{2}+e_{3}, \\ 
e_{2}\cdot e_{1}=e_{3},e_{1}\cdot e_{3}=e_{3}%
\end{array}
\right. $ & $\mathcal{G}_{3,2}$ & D \\ \hline
$C_{3,1}$ & $e_{1}\cdot e_{2}=e_{2}+e_{3},$ $e_{1}\cdot e_{3}=e_{3}$ & $%
\mathcal{G}_{3,3}$ & N, D, S \\ \hline
$C_{3,t}$ & $\left. 
\begin{array}{c}
e_{1}\cdot e_{2}=e_{2}+te_{3},\text{ }e_{1}\cdot e_{3}=e_{3}, \\ 
e_{2}\cdot e_{1}=\left( t-1\right) e_{3},,t\neq 1%
\end{array}
\right. $ & $\mathcal{G}_{3,3}$ & D \\ \hline
$D_{3,1}\left( \mu \right) $ & $\left. 
\begin{array}{c}
e_{1}\cdot e_{2}=e_{2}, \\ 
e_{1}\cdot e_{3}=\mu e_{3},\text{ }0<\left| \mu \right| <1%
\end{array}
\right. $ & $\mathcal{G}_{3,4}^{\mu }$ & N, D, S \\ \hline
$D_{3,2}$ & $\left. 
\begin{array}{c}
e_{1}\cdot e_{2}=e_{2},e_{1}\cdot e_{3}=\frac{1}{2}e_{3}, \\ 
e_{2}\cdot e_{2}=e_{1}%
\end{array}
\right. $ & $\mathcal{G}_{3,4}^{\frac{1}{2}}$ & N \\ \hline
$E_{3,1}\left( \zeta \right) $ & $\left. 
\begin{array}{c}
e_{1}\cdot e_{2}=e_{2}+\zeta e_{3}, \\ 
e_{1}\cdot e_{3}=-\zeta e_{2}+e_{3},\text{ }\zeta >0%
\end{array}
\right. $ & $\mathcal{G}_{3,5}^{\zeta }$ & N, D, S \\ \hline
\end{tabular}%
\end{equation*}

Here, the letter N means that the left-symmetric algebra $A_{3}$ is Novikov,
the letter D means that $A_{3}$ is derivation and the letter S means that $%
A_{3}$ satisfying $\left[ x,y\right] \cdot z=0$ for all $x,y,z\in $ $A_{3}.$
\end{theorem}

\begin{remark}
We note that left-symmetric algebras satisfying the identity $\left( x\cdot
y\right) \cdot z=\left( y\cdot x\right) \cdot z\ \ $for all $x,y,z\in A$ (or
equivalently, the identity $\left[ x,y\right] \cdot z=0\ \ \ $for all $%
x,y,z\in A$) are of special interest because they correspond to locally
simply transitive affine actions of Lie groups $G$ on a vector space $E$
such that the commutator subgroup $\left[ G,G\right] $\ is acting by
translations. These left-symmetric algebras have been considered and studied
in \cite{guediri}.
\end{remark}

We note that the mapping $X\rightarrow \left( L_{X},X\right) $ is a Lie
algebra representation of $\mathcal{G}$ in aff$\left( \mathbb{R}^{3}\right)
=End\left( \mathbb{R}^{3}\right) \bigoplus \mathbb{R}^{3}.$ By using the
exponential maps, Theorem \ref{complete3} can now be stated, in terms of
simply transitive actions of subgroups of the affine group $Aff(\mathbb{R}%
^{3})=GL\left( \mathbb{R}^{3}\right) \mathbb{o}\mathbb{R}^{3},$ as follows

To state it, define the continuous functions $f,g,h,k$ and $\phi $ by 
\begin{eqnarray*}
f\left( x\right) &=&\left\{ 
\begin{tabular}{cc}
$\frac{e^{x}-1}{x},$ & $x\neq 0$ \\ 
$1,$ & $x=0$%
\end{tabular}
\right. ,\qquad \qquad g\left( x\right) =\left\{ 
\begin{tabular}{ll}
$\frac{e^{x}-x-1}{x^{2}},$ & $x\neq 0$ \\ 
$\frac{1}{2}$ & $x=0$%
\end{tabular}
\right. \\
h\left( x\right) &=&\left\{ 
\begin{tabular}{cc}
$\frac{\cos x-1}{x}+\frac{x}{2},$ & $x\neq 0$ \\ 
$0,$ & $x=0$%
\end{tabular}
\right. ,\qquad k\left( x\right) =\left\{ 
\begin{tabular}{cc}
$\frac{\sin x-x}{x},$ & $x\neq 0$ \\ 
$0,$ & $x=0$%
\end{tabular}
\right. \\
\phi \left( x\right) &=&\sum_{n=1}^{\infty }\frac{nx^{n}}{\left( n+1\right) !%
}
\end{eqnarray*}

\begin{theorem}
\label{sim action}Suppose that the Lie group $G$ of the non-unimodular Lie
algebra $\mathcal{G}$ of dimension 3 acts simply transitively by affine
transformations on $\mathbb{R}^{3}.$ Then, as a subgroup of $Aff(\mathbb{R}%
^{3}),$ $G$ is conjugate to one of the following subgroups:

\begin{equation*}
\begin{tabular}{l}
$G_{A_{3,0}}=\left\{ \left( 
\begin{array}{lll}
1 & 0 & 0 \\ 
0 & e^{a} & 0 \\ 
0 & 0 & 1%
\end{array}%
\right) \left[ 
\begin{array}{l}
a \\ 
bf\left( a\right) \\ 
c%
\end{array}%
\right] ,\text{ }a,b,c\in \mathbb{R}\right\} $ \\ 
$G_{A_{3,1}}=\left\{ \left( 
\begin{array}{lll}
1 & 0 & 0 \\ 
0 & e^{a} & 0 \\ 
a & 0 & 1%
\end{array}%
\right) \left[ 
\begin{array}{l}
a \\ 
bf\left( a\right) \\ 
c+\frac{1}{2}a^{2}%
\end{array}%
\right] ,\text{ }a,b,c\in \mathbb{R}\right\} $ \\ 
$G_{A_{3,2}}=\left\{ \left( 
\begin{array}{lll}
1 & 0 & c \\ 
0 & e^{a} & 0 \\ 
0 & 0 & 1%
\end{array}%
\right) \left[ 
\begin{array}{l}
a+\frac{1}{2}c^{2} \\ 
bf\left( a\right) \\ 
c%
\end{array}%
\right] ,\text{ }a,b,c\in \mathbb{R}\right\} $ \\ 
$G_{A_{3,3}}=\left\{ \left( 
\begin{array}{lll}
1 & 0 & -c \\ 
0 & e^{a} & 0 \\ 
0 & 0 & 1%
\end{array}%
\right) \left[ 
\begin{array}{l}
a-\frac{1}{2}c^{2} \\ 
bf\left( a\right) \\ 
c%
\end{array}%
\right] ,\text{ }a,b,c\in \mathbb{R}\right\} $ \\ 
$G_{B_{3,0}}=\left\{ \left( 
\begin{array}{lll}
1 & 0 & 0 \\ 
0 & e^{a} & 0 \\ 
0 & 0 & e^{a}%
\end{array}%
\right) \left[ 
\begin{array}{l}
a \\ 
bf\left( a\right) \\ 
cf\left( a\right)%
\end{array}%
\right] ,\text{ }a,b,c\in \mathbb{R}\right\} $ \\ 
$G_{B_{3,1}}=\left\{ \left( 
\begin{array}{lll}
1 & 0 & 0 \\ 
0 & e^{a} & 0 \\ 
bf\left( a\right) & ae^{a} & e^{a}%
\end{array}%
\right) \left[ 
\begin{array}{l}
a \\ 
bf\left( a\right) \\ 
\left( ab+c\right) f\left( a\right)%
\end{array}%
\right] ,\text{ }a,b,c\in \mathbb{R}\right\} $ \\ 
$G_{C_{3,1}}=\left\{ \left( 
\begin{array}{lll}
1 & 0 & 0 \\ 
0 & e^{a} & 0 \\ 
0 & ae^{a} & e^{a}%
\end{array}%
\right) \left[ 
\begin{array}{l}
a \\ 
bf\left( a\right) \\ 
cf\left( a\right) +b\phi \left( a\right)%
\end{array}%
\right] ,\text{ }a,b,c\in \mathbb{R}\right\} $ \\ 
$G_{C_{3,t}}=\left\{ \left( 
\begin{array}{lll}
1 & 0 & 0 \\ 
0 & e^{a} & 0 \\ 
\left( t-1\right) bf\left( a\right) & tae^{a} & e^{a}%
\end{array}%
\right) \left[ 
\begin{array}{l}
a \\ 
bf\left( a\right) \\ 
\left( tab+c-b\right) f\left( a\right) +b%
\end{array}%
\right] ,\text{ }a,b,c\in \mathbb{R},\text{ }t\neq 1\right\} $ \\ 
$G_{D_{3,1}\left( \mu \right) }=\left\{ \left( 
\begin{array}{lll}
1 & 0 & 0 \\ 
0 & e^{a} & 0 \\ 
0 & 0 & e^{\mu a}%
\end{array}%
\right) \left[ 
\begin{array}{l}
a \\ 
bf\left( a\right) \\ 
cf\left( \mu a\right)%
\end{array}%
\right] ,\text{ }a,b,c\in \mathbb{R}\right\} \text{ },0<\left\vert \mu
\right\vert <1$ \\ 
$G_{D_{3,2}}=\left\{ \left( 
\begin{array}{lll}
1 & bf\left( a\right) & 0 \\ 
0 & e^{a} & 0 \\ 
0 & 0 & e^{\frac{1}{2}a}%
\end{array}%
\right) \left[ 
\begin{array}{l}
a+b^{2}g\left( a\right) \\ 
bf\left( a\right) \\ 
cf\left( \frac{a}{2}\right)%
\end{array}%
\right] ,\text{ }a,b,c\in \mathbb{R}\right\} $ \\ 
$G_{E_{3}\left( \zeta \right) }=\left\{ 
\begin{array}{c}
\left( 
\begin{array}{lll}
1 & 0 & 0 \\ 
0 & e^{a}\cos \zeta a & -e^{a}\sin \zeta a \\ 
0 & e^{a}\sin \zeta a & e^{a}\cos \zeta a%
\end{array}%
\right) \\ 
\left[ 
\begin{array}{l}
a \\ 
b\left( f\left( a\right) +k\left( \zeta a\right) \right) +c\left( h\left(
\zeta a\right) -\zeta \phi \left( a\right) \right) \\ 
b\left( \zeta \phi \left( a\right) -h\left( \zeta a\right) \right) +c\left(
f\left( a\right) +k\left( \zeta a\right) \right)%
\end{array}%
\right] ,\text{ }a,b,c\in \mathbb{R}\text{,}\zeta >0%
\end{array}%
\right\} $%
\end{tabular}%
\end{equation*}
\end{theorem}

\pagebreak

Department of Mathematics,

College of Science,

King Saud University,

P.O.Box 2455, Riyadh 11451

Saudi Arabia

\smallskip

\bigskip

E-mail addresses: mguediri@ksu.edu.sa (M. Guediri),

~~~$~~$\ \ \ \ \ \ \ \ \ \ \ \ \ \ \ \ \ \ \ \ \ ksf\_math@hotmail.com (K.
Al-Balawi).

\end{document}